\documentclass[english,15pt]{article}
\usepackage{hyperref}
\usepackage[T1]{fontenc}
\usepackage[latin1]{inputenc}
\usepackage{geometry}
\usepackage{float}
\usepackage{mathrsfs}
\usepackage{amsmath}
\usepackage{amssymb}
\usepackage{graphicx}
\usepackage{subcaption}
\usepackage{booktabs}
\usepackage{multirow}

\hypersetup{
    colorlinks=true,
    linkcolor=blue,
    filecolor=magenta,      
    urlcolor=cyan,
    citecolor = blue,
}

\makeatletter

\usepackage{algorithm}
 \usepackage{algorithmicx}

\usepackage{todonotes}

\usepackage{algorithm,algpseudocode}

\usepackage{amsthm}

\usepackage{mathrsfs}

\usepackage{amsfonts}

\usepackage{epsfig}

\usepackage{bm}

\usepackage{mathrsfs}

\usepackage{enumerate}

\numberwithin{equation}{section}

\@ifundefined{definecolor}{\@ifundefined{definecolor}
 {\@ifundefined{definecolor}
 {\usepackage{color}}{}
}{}
}{}

\newtheorem{prop}{Proposition}[section]

\newcounter{hypA}
\newenvironment{hypA}{\refstepcounter{hypA}\begin{itemize}
  \item[({\bf A\arabic{hypA}})]}{\end{itemize}}
\newcounter{hypB}

\newcounter{hypD}

\usepackage{babel}\date{}

\usepackage{babel}

\makeatother

\usepackage{babel}

\begin{document}

\begin{center}

{\Large \textbf{Unbiased Gradients for a Class of Conditional  \\ Stochastic Optimization Problems}}

\vspace{0.5cm}

BY  MIGUEL ALVAREZ$^{1}$  \& AJAY JASRA$^{2}$

{\footnotesize $^{1}$Applied Mathematics and Computational Science Program, Computer, Electrical and Mathematical Sciences and Engineering Division, King Abdullah University of Science and Technology, Thuwal, 23955-6900, KSA.}\\
{\footnotesize $^{2}$School of Data Science,  The Chinese University of Hong Kong,  Shenzen,  Shenzhen, CN.}\\
{\footnotesize E-Mail:\,} \texttt{\emph{\footnotesize miguelangel.alvarezballesteros@kaust.edu.sa}}, \\\texttt{\emph{\footnotesize ajayjasra@cuhk.edu.cn}}
\end{center}

\begin{abstract}
In this paper we consider the conditional stochastic optimization (CSO) problem.  This consists
of optimizing a function which can be written as the expectation of a function which is itself a function
of a conditional expectation,  i.e.~of the type $F(\xi) := \mathbb{E}\left[f\left(Z,\mathbb{E}[g(Z,X,\xi)|Z]\right)\right]$, where precise definitions are given in the main text.  We address a particular class of CSO problems where the joint law of the random variables $X,Z$ cannot be exactly sampled; this case has been addressed in \cite{goda}.  We introduce a method that combines Markovian stochastic approximation with unbiased approximation methods which allows one to find the optimizer of $F(\xi)$ in the context of interest.  We illustrate our methodology on two examples associated to parameter estimation with model averaging and portfolio selection associated to high-dimensional full factor multivariate stochastic volatility models.
\end{abstract}

\section{Introduction}

We are given two random variables $(Z,X)\in\mathsf{Z}\times\mathsf{X}$ defined on a probability space $(\Omega,\mathcal{F},\mathbb{P})$ and a collection of parameters $\xi\in\Xi\subseteq\mathbb{R}^{k}$. Let $g:\mathsf{Z}\times\mathsf{X}\times\Xi\rightarrow\mathbb{R}^k$ and $f:\mathsf{Z}\times\mathbb{R}^k\rightarrow\mathbb{R}$ be appropriately
measurable.  
 The problem of interest is to compute the following optimal value(s):
$$
\min_{\xi\in\Xi} F(\xi) \quad \textrm{where} \quad F(\xi) := \mathbb{E}\left[f\left(Z,\mathbb{E}[g(Z,X,\xi)|Z]\right)\right].
$$
This is the conditional stochastic optimization (CSO) problem; see e.g.~\cite{hu,hu1}. The CSO problem  can be challenging primarily due to the computation of $F(\xi)$, the objective function, which has 
both an inner and outer expectation. 
Under standard assumptions, the gradient of $F$ w.r.t.~$\xi$ is 
$$
\nabla F(\xi) = \mathbb{E}\left[\mathbb{E}[\nabla g(Z,X,\xi)|Z]^{\top}\nabla f\left(Z,\mathbb{E}[g(Z,X,\xi)|Z]\right)\right].
$$
The objective is to compute an unbiased estimate of $\nabla F(\xi)$. This motivates the use of unbiased Monte Carlo for stochastic optimization problems.  The CSO problem has many applications such as in regression or meta-learning; see for instance the articles \cite{hu1, singh}.

The problem of unbiased estimation of $\nabla F(\xi)$ has been considered in \cite{goda} (see also \cite{goda2}) under the specific scenario that exact simulation from the joint law of $(Z,X)$ is possible.  We shall assume that the joint law of $(Z,X)$ will admit a density $\pi(z)\nu(x|z)$ w.r.t.~a dominating measure $dx\otimes dz$ and shall make the constraint that this density is known point-wise up-to a normalizing constant.
We shall consider the more complex problem where exact simulation of  $(Z,X)$  is not necessarily available and one must resort to Markov chain Monte Carlo (MCMC) methods,  that is,  one needs more advanced sampling techniques than independent sampling.   In this context there are a few methods that one could pursue,  for instance one is to find an exact unbiased estimator of $\nabla F(\xi)$ and plug that directly into a stochastic gradient (SG) method,  which is the typical approach; see for instance the work of \cite{ub_grad,ub_grad_bip} where one has to use quite advanced MCMC methods for some challenging model problems.  Whilst these type of methods can deliver exact,  up-to the constraints of SG methods, optimizers of $F$ the computation of an unbiased estimator of the gradient,  especially for the model problems we present is rather expensive.

The idea we follow is based upon the works of \cite{ub_par,ub_filt}.  In the article of \cite{ub_par} the authors recognized that for a certain type of optimization problem associated to diffusions and time-discretizations,  where exact simulation is not feasible,  one can adopt the well-known idea of Markovian stochastic approximation (MSA) (e.g.~\cite{andr_moulines1,andr}) and the bias (there time discretization) can still be removed.  MSA
simply uses a Markov kernel to facilitate a SG type method for optimization (amongst other things).
To obtain a trivially valid optimization scheme we then use the approach of \cite{ub_filt}
to obtain an unbiased estimate of 
$$
\mathbb{E}[\nabla g(Z,X,\xi)|Z]^{\top}\nabla f\left(Z,\mathbb{E}[g(Z,X,\xi)|Z]\right).
$$
In other words the $Z$ variable is sampled with MCMC and the inner term above is unbiasedly estimated at each time step of an MSA procedure. 
We prove that, under assumptions,  the estimator is indeed unbiased and of finite variance.  We investigate our methodology in two challenging problems.  The first is parameter estimation associated to state-space models 
(e.g.~\cite{cappe}) with model averaging.   The second problem is associated to portfolio selection when modeling financial data with a high-dimensional full factor multivariate stochastic volatility model as in \cite{mv_sv}.

This article is structured as follows. In Section \ref{sec:approach} our methodology is detailed along with a mathematical result on the unbiasedness and finite variance of our estimator.
In Section \ref{sec:numerics} we provide numerical examples which demonstrate our methodology. 
In Appendix \ref{app:prf_prop} the proof of our mathematical result in Section \ref{sec:approach} is given.

\section{Approach}\label{sec:approach}

\subsection{Markovian Stochastic Approximation}

The basic mechansim that we will use in Markovian stochastic approximation (MSA).  At this stage we shall assume that there exists an unbiased estimate of 
$$
H(Z,\xi) : = \mathbb{E}[\nabla g(Z,X,\xi)|Z]^{\top}\nabla f\left(Z,\mathbb{E}[g(Z,X,\xi)|Z]\right)
$$
denoted as $\widehat{H}(Z,\xi)$.
If one can sample exactly from $\nu(x|z)$ then there is a method in \cite{goda}.  Alternatively,  if there is only an MCMC method available then we detail an approach below.

Let $K:\mathsf{Z}\times\mathscr{Z}\rightarrow[0,1]$ be any $\pi-$stationary and ergodic Markov kernel,  where $\pi(z)dz$ is a probability measure on the measurable space $(\mathsf{Z},\mathscr{Z})$.  Let $(z_0,\xi_0)\in\mathsf{Z}\times \Xi$ be given, then the MSA method is as follows starting with $n=1$ and $\{\gamma_n\}_{n\geq 0}$ standard learning rate parameters that are used in MSA.
\begin{enumerate}
\item{Sample $Z_n|Z_{n-1}\sim K(z_{n-1},\cdot)$.}
\item{Update:
$$
\xi_n = \xi_{n-1} + \gamma_n\widehat{H}(z_n,\xi_{n-1})
$$
Set $n=n+1$ and return to the start of 1..}
\end{enumerate}

We remark that the above approach is essentially a regular MSA method,  as the randomness we will use to 
compute $\widehat{H}(Z,\xi)$ depends on $Z,\xi$ and generated in an independent fashion of all other random variables in the algorithm.  As such the convergence theory for this approach is the standard one,  that can be found for instance in \cite{andr_moulines1}.

\subsection{Unbiased Estimate of $H(Z,\xi)$}
\label{subsec:alg}
For any $z\in\mathsf{Z}$ let $\eta(x|z)dx$ be a probability measure on $(\mathsf{X},\mathscr{X})$ where $\eta$ is strictly positive and $M_z:\mathsf{X}\times\mathscr{X}\rightarrow[0,1]$  be an ergodic Markov kernel that admits $\nu(x|z)dx$ as its stationary measure.
 For $N\in\mathbb{N}$ consider the following procedure:
\begin{enumerate}
\item{Sample $X_0|z\sim \eta(\cdot|z)$. Set $n=1$.}
\item{Sample $X_n|x_{n-1},z\sim M_z(x_{n-1},\cdot)$. Set $n=n+1$ and if $n=N+1$ stop and return
$$
\nu_z^N(dx) := \frac{1}{N+1}\sum_{n=0}^{N}\delta_{\{x_n\}}(dx)
$$
otherwise return to the start of 2..}
\end{enumerate}
We call the above procedure \textbf{MC}.  We use the notation for any function $\alpha:\mathsf{X}\rightarrow\mathbb{R}^d$ (to be read as a column vector):
$$
\nu_z^N(\alpha) = \frac{1}{N+1}\sum_{n=0}^{N}\alpha(x_n).
$$

To provide an unbiased estimate $H(Z,\xi)$ we can follow the approach in \cite{ub_filt}.
Let $N_l=2^{l}$,  $l\in\mathbb{N}$,  $N_0=0$ and $\mathbb{P}_{\texttt{L}}(l)$ be a positive probability on
$\mathbb{N}$.  We have the following method.
\begin{enumerate}
\item{Sample $L\sim\mathbb{P}_{\texttt{L}}(\cdot)$.}
\item{If $L=1$ then run \textbf{MC} with $N_1$ samples
and return
$$
\widehat{H}(Z,\xi) = \frac{1}{\mathbb{P}_{\texttt{L}}(L)}
\nu_z^{N_1}(\nabla g(Z,\cdot,\xi))^{\top}\nabla f\left(Z,\nu_z^{N_1}(g(Z,\cdot,\xi))\right)
$$
Otherwise go to 3..}
\item{For $l\in\{1,\dots,L\}$ independently run \textbf{MC} with $N_l-N_{l-1}$ samples.
Set for any $l\in\mathbb{N}$ and any function $\alpha:\mathsf{X}\rightarrow\mathbb{R}^d$ 
$$
\nu_z^{N_{1:l}}(\alpha) := \sum_{p=1}^l \left(\frac{N_p-N_{p-1}+1}{N_l}\right)\nu_z^{N_p-N_{p-1}}(\alpha).
$$
Then return
$$
\widehat{H}(Z,\xi) = 
$$
$$
\frac{1}{\mathbb{P}_{\texttt{L}}(L)}\left\{
\nu_z^{N_{1:L}}(\nabla g(Z,\cdot,\xi))^{\top}\nabla f\left(Z,\nu_z^{N_{1:L}}(g(Z,\cdot,\xi))\right) - 
\nu_z^{N_{1:L-1}}(\nabla g(Z,\cdot,\xi))^{\top}\nabla f\left(Z,\nu_z^{N_{1:L-1}}(g(Z,\cdot,\xi))\right)
\right\}.
$$
}
\end{enumerate}

\subsection{Theory}

We will establish that $\widehat{H}(Z,\xi)$ is indeed an unbiased estimator.  We will make the following assumption.
Note that for any vector $x\in\mathbb{R}^d$ then we write $x^{(i)}$ as its $i^{\text{th}}$ component,  $i\in\{1,\dots,d\}$,  similarly for any matrix $x$ that is $d_1\times d_2$ we write $x^{(i,j)}$ as its $(i,j)^{\text{th}}$ entry  $(i,j)\in\{1,\dots,d_1\}\times\{1,\dots,d_2\}$.

\begin{hypA}\label{ass:1}
\begin{enumerate}
\item{There exists $\varepsilon\in(0,1)$ and a probability measure $\kappa$ on $(\mathsf{X},\mathscr{X})$ such that for
any $(z,x)\in\mathsf{Z}\times\mathsf{X}$
$$
M_z(x,dx') \geq \varepsilon\kappa(dx').
$$
}
\item{We have that for $(i,j)\in\{1,\dots,k\}^2$:
\begin{align*}
\sup_{(z,x,\xi)\in\mathsf{Z}\times\mathsf{X}\times\Xi}|g(z,x,\xi)^{(i)}| & <+\infty\\
\sup_{(z,u)\in\mathsf{Z}\times\mathbb{R}^k}|\nabla f(z,u)^{(i)}| & <+\infty \\
\sup_{(z,x,\xi)\in\mathsf{Z}\times\mathsf{X}\times\Xi}|\nabla g(z,x,\xi)^{(i,j)}| & <+\infty.
\end{align*}
}
\item{There exists a $C<\infty$ such that for any $(i,z,z',u,u'))\in\{1,\dots,k\}\times\mathsf{Z}^2\times\mathbb{R}^{2k}$:
$$
|\nabla f(z,u)^{(i)}-\nabla f(z,u')^{(i)}|  < C\left(\|z-z'\|+\|u-u'\|\right)
$$
where $\|\cdot\|$ is the $L_2-$norm.}
\end{enumerate}
\end{hypA}

The proof of the following result can be found in Appendix \ref{app:prf_prop}. Below $\mathbb{E}$ denotes expectation w.r.t.~the law used to generate $\widehat{H}(Z,\xi)$.

\begin{prop}\label{prop:1}
Assume (A\ref{ass:1}) and for any $\beta\in(0,1)$ set $\mathbb{P}_{\texttt{L}}(l)\propto 2^{-\beta l}$.  Then
for any $(i,z,\xi)\in \{1,\dots,k\}\times\mathsf{Z}\times\Xi$ we have that 
\begin{align*}
\mathbb{E}[\widehat{H}(z,\xi)] & = H(z,\xi) \\
\mathbb{E}[\left(\widehat{H}(Z,\xi)^{(i)}\right)^2] & < +\infty.
\end{align*}
\end{prop}

Proposition \ref{prop:1} establishes the unbiased and finite variance property of the estimator of $H(Z,\xi)$.  Just as in \cite{ub_filt} the expected cost is infinite to obtain the estimator,  but none-the-less,  with high probability we can find the estimator with finite computational cost.  We refer the reader to \cite{ub_filt} for the details.

\section{Numerical Examples}\label{sec:numerics}

In this section, we present two numerical applications of the proposed unbiased methodology. 
The first example considers parameter estimation in a state-space model with a mixture observation likelihood, combining Student's $t$ distributions and a Gaussian component. 
The second example focuses on a portfolio selection problem based on multivariate stochastic volatility (MSV) models, following the framework of \cite{mv_sv}. 
All algorithms are implemented in MATLAB, and the corresponding code is available at \url{https://github.com/maabs/Unbiased-CSO}.

\subsection{Model Averaged Parameter Estimation for State-Space Models}

\subsubsection{Modeling}

Consider the following state-space model for $X_0=x_0\in\mathbb{R}$ given and $n\in\mathbb{N}$:
\begin{align*}
X_n  & = T_{n,\theta}(X_{n-1},\epsilon_n) \\
M & \stackrel{\text{ind}}{\sim} \mathcal{U}_{\{1,\dots,m,m^{\star}\}}\\
Y_n|X_n,M & \stackrel{\text{ind}}{\sim} \
\left\{\begin{array}{ll}
\mathcal{T}_M(X_n,\sigma^2) & \text{if}~M\in\{1, \dots,m\}\\
\mathcal{N}(X_n,\sigma^2) & \text{if}~M=m^{\star}
\end{array}\right.
\end{align*}
where 
$\epsilon_n$ is an i.i.d.~real sequence (independent of all other random variables in the model) with a positive Lebesgue density,  $T_{n,\theta}:\mathbb{R}^2\rightarrow\mathbb{R}$ is such that the conditional Lebesgue density of $X_n$ on $X_{n-1}$, $t_{n,\theta}(x_n|x_{n-1})$ and its gradient $\nabla t_{n,\theta}(x_n|x_{n-1})$ w.r.t. the static parameters are available, $\nabla$ denotes the gradient w.r.t. the parameters $\xi$,  
$\stackrel{\text{ind}}{\sim}$ means a random variable is independently distributed as, 
$\mathcal{U}_{\{1,\dots,m,m^{\star}\}}$ is the discrete uniform on $\mathsf{M}:=\{1,\dots,m,m^{\star}\}$,  $\mathcal{T}_M(X_n,\sigma^2)$ is the Student-$t$ distribution with $M$ degrees of freedom, location $X_n$ and scale $\sigma$, and $\mathcal{N}(X_n,\sigma^2)$ is the normal distribution of mean $X_n$ and variance $\sigma^2$.

In this problem, denoting the conditional density of $y_n|x_n,m$  as $g_{\sigma}(y_n|x_n,m) $,  we seek to estimate $\xi=(\sigma,\theta)$, $\theta\in\Theta$,  $\Xi=\mathbb{R}^+\times\Theta$ so as to maximize the marginal likelihood for $\mathscr{T}\in\mathbb{N}$ fixed,  
$$
\sup_{\xi\in\Xi}\sum_{m\in\mathsf{M}} \int_{\mathbb{R}^{\mathscr{T}}} \left\{
\prod_{n=1}^{\mathscr{T}}g_{\sigma}(y_n|x_n,m) t_{n,\theta}(x_n|x_{n-1})  \right\}dx_{1:\mathscr{T}} p(m)
$$
where $p(m)$ is the prior on $m$.  The problem is equivalent to
$$
\sup_{\xi\in\Xi}  \log\left(
\sum_{m\in\mathsf{M}} \int_{\mathbb{R}^{\mathscr{T}}} \left\{
\prod_{n=1}^\mathscr{T}g_{\sigma}(y_n|x_n,m) t_{n,\theta}(x_n|x_{n-1})\right\} dx_{1:\mathscr{T}} p(m)
\right)
$$
with gradient:
$$
 \sum_{m\in\mathsf{M}} \int_{\mathbb{R}^{\mathscr{T}}} 
\left\{
\sum_{n=1}^{\mathscr{T}} \nabla \log
\left(
g_{\sigma}(y_n|x_n,m) t_{n,\theta}(x_n|x_{n-1})
\right)
\right\}
\pi_{\theta}(x_{1:\mathscr{T}}|y_{1:\mathscr{T}},m) dx_{1:\mathscr{T}}p(m)
$$
where
$$
\pi_{\theta}(x_{1:\mathscr{T}}|y_{1:\mathscr{T}},m) \propto \prod_{n=1}^{\mathscr{T}}g_{\sigma}(y_n|x_n,m) t_{n,\theta}(x_n|x_{n-1}).
$$
We remark that averaging over $M$ and then having to perform an MCMC step for each invariant measure,  just to compute the gradient is too expensive.  So one can sample $m$ and then use our methodology to compute an unbiased estimate of the gradient.

\subsubsection{Numerical Settings and Results}
\label{subsubsec:num_set_first_example}
 The choice of the hidden model of the state-space is a Gaussian/linear model for $n\in\{1,\dots,\mathscr{T}\}$
\[
X_n = \mu X_{n-1} + {\Sigma}^{\frac{1}{2}}\,\epsilon_n, 
\]
where $\theta = (\mu,\Sigma)$ with $\mu \in (-1,1)$ and $\Sigma \in \mathbb{R}^+$, and 
$\epsilon_n\stackrel{\text{i.i.d.}}{\sim}\mathcal{N}(0,1)$, $x_0 = 0$ and $\mathscr{T} = 30$.
We generate data from the model with $\xi = (\sigma,\mu,\Sigma) = (0.3, 0.95,0.2)$.

Let $N_l = 2^{\,l}$, $l\in \mathbb{N}$, and let $q\geq 1$. We define the level distribution of the unbiased estimator as
\[
\mathbb{P}_{\mathtt{L}}(l)\propto (l+q)\,\log^{2}(l+q)\,N_l^{-1},
\]
which ensures unbiasedness and finite variance of the estimator.
In practice, truncation of the level distribution is necessary to bound the maximum computational time of the estimator; moreover, truncation prevents potential memory overload. We therefore consider the truncated distribution
\begin{align}
    \widehat{\mathbb{P}}_{\mathtt{L}}(l)\propto (l+q)\,\log^{2}(l+q)\,N_l^{-1},
\qquad l \in \{0,1,\dots,L_{\max}\},
\label{eq:trunc_prob}
\end{align}
with $L_{\max}=10$ and $q=4$. The choice of $L_{\max}$ is sufficient to ensure that the resulting estimator is effectively unbiased within the target error regime.

We use the particle Gibbs algorithm with backward sampling \cite{ljs2014}
to  sample of $\pi_\theta(x_{1:T}\mid y_{1:T},m)$.  
Particle Gibbs is known to exhibit good mixing properties when the number of particles scales as $N_{pf}=\mathcal{O}(\mathscr{T}^{\delta})$, with $\delta \geq 1$ \cite{ldm2014}. In our implementation, we fix $N_{pf}=10$, which results in a good mixing rate and a relatively low computational cost per MCMC iteration.

Consider the average of $S\in \mathbb{N}$ independent unbiased estimators,
$$
 H_S(m,\xi) =
\frac{1}{S} \sum_{s=1}^S \widehat H^{(s)}(m,\xi).
$$
Figures~\ref{fig:relMSE_gaus} and~\ref{fig:relMSEtstud} display an approximation of the relative mean squared error (MSE) of this averaged estimator \( H_S \), under different observation models. Figure~\ref{fig:relMSE_gaus} corresponds to the case where we have Gaussian observations, whereas Figure~\ref{fig:relMSEtstud} corresponds to the case where we have Student-$t$ distributed observations.
The Student-$t$ observation model converges in distribution to the Gaussian model as 
$m \to \infty$. 
In order to estimate the MSE, we compare against the analytical score of the linear-Gaussian 
state-space model.  
We use a very large degrees-of-freedom parameter, $m = 10^5$, so that the score of the 
Student-$t$ observation model is, for practical purposes, indistinguishable 
from the Gaussian score.
Our objective when constructing the truncated unbiased estimator is to choose $L_{\max}$ so that 
the squared bias is several orders of magnitude smaller than the target MSE; this is the case for both figures~\ref{fig:relMSE_gaus} and ~\ref{fig:relMSEtstud}.

\begin{figure}[htbp]
    \centering
    
    \begin{subfigure}[t]{0.45\textwidth}
        \centering
        \includegraphics[width=\linewidth]{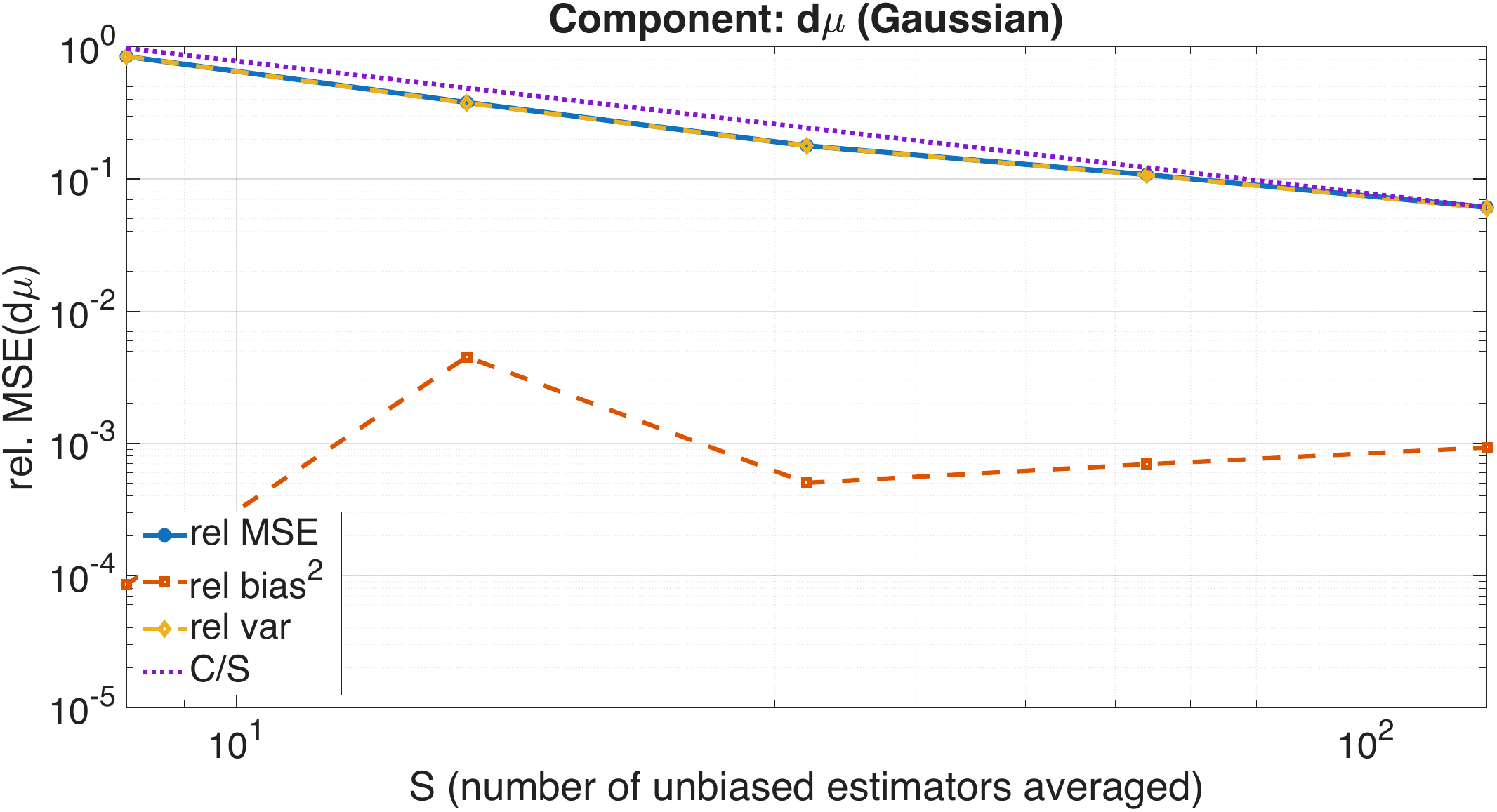}
        \label{fig:cn1}
    \end{subfigure}
    \hfill
    \begin{subfigure}[t]{0.45\textwidth}
        \centering
        \includegraphics[width=\linewidth]{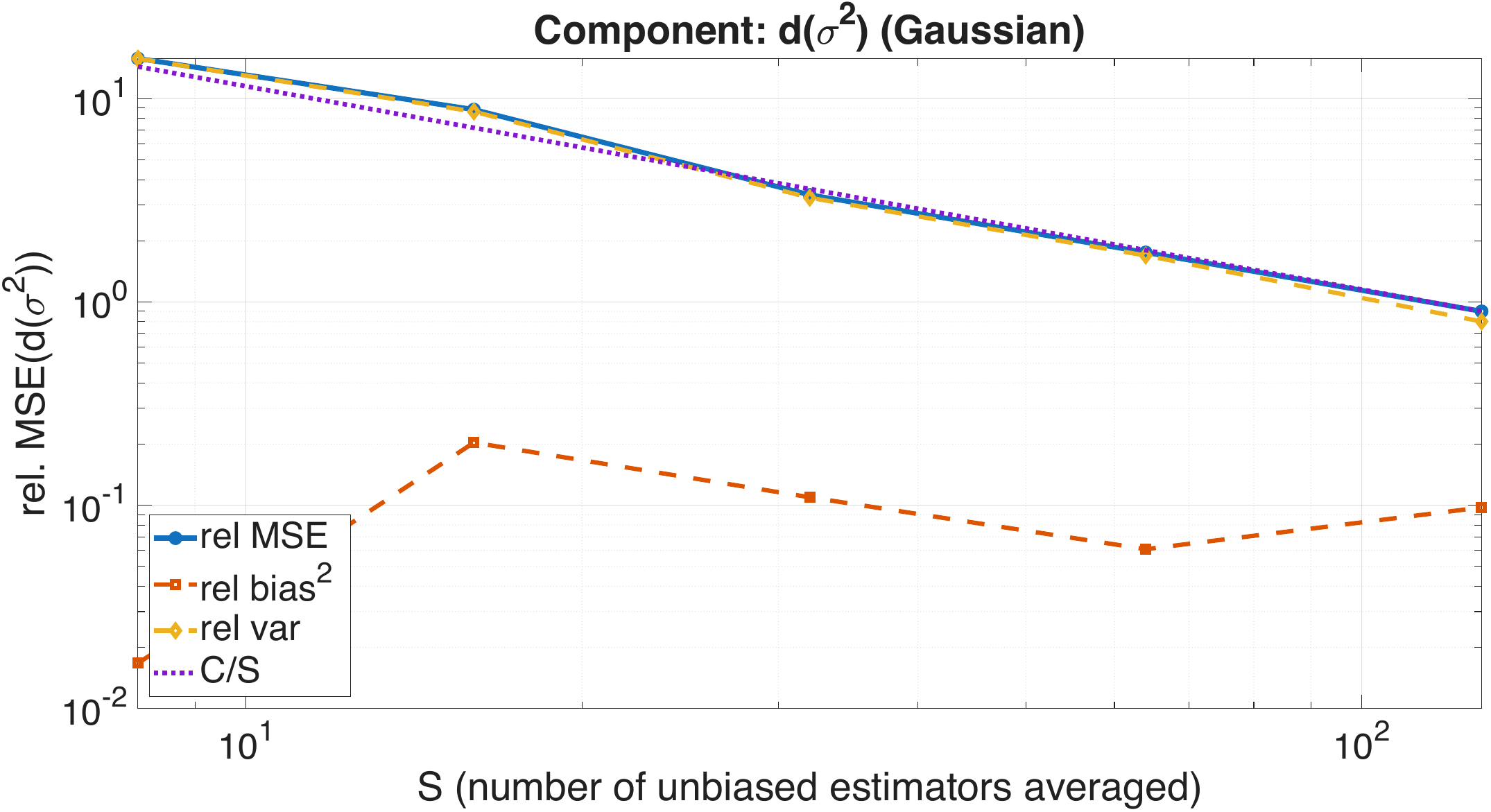}
        
        \label{fig:cn2}
    \end{subfigure}

    \vspace{0.5cm}

    \begin{subfigure}[t]{0.45\textwidth}
        \centering
        \includegraphics[width=\linewidth]{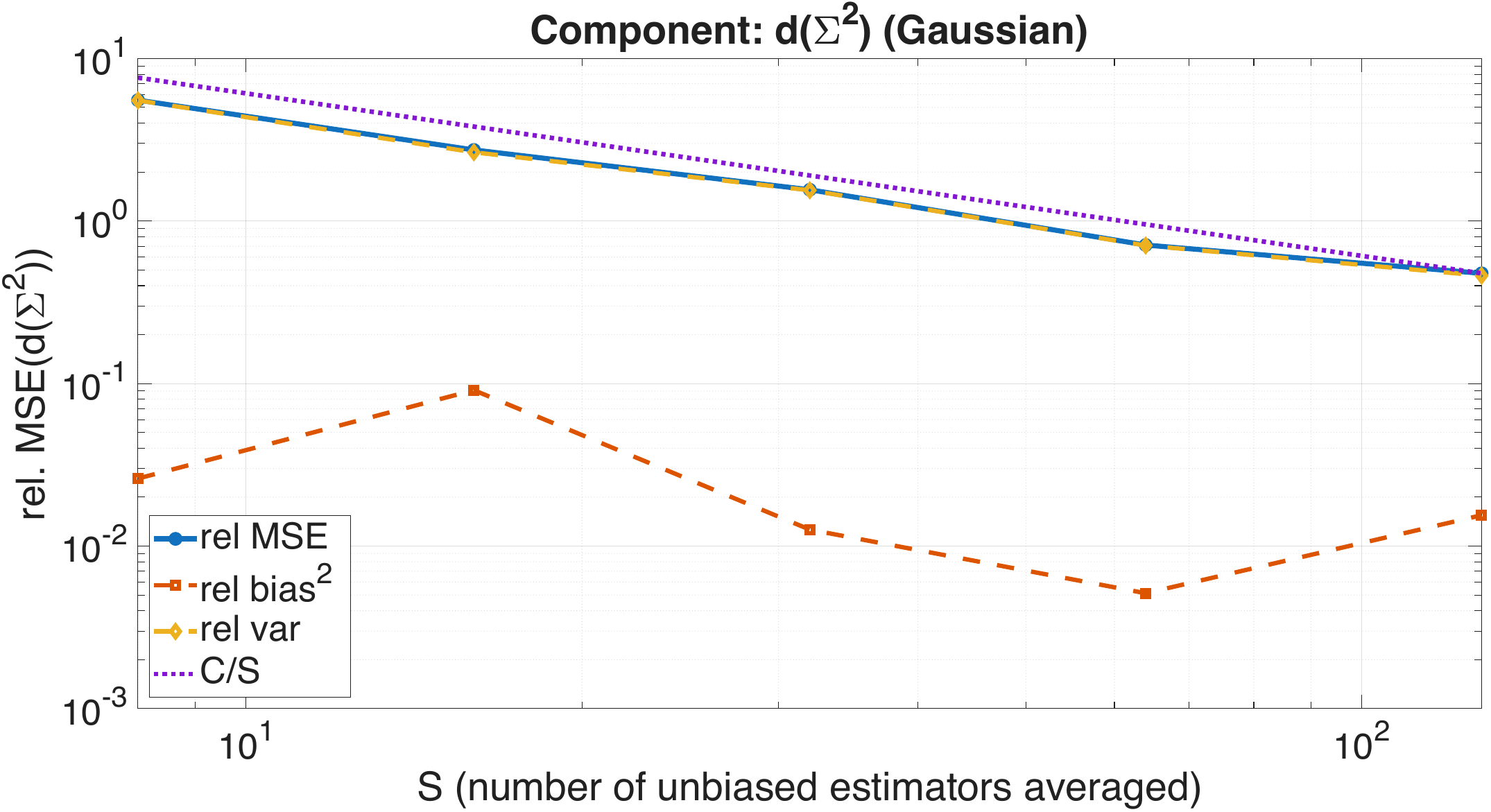}
        
        \label{fig:cn3}
    \end{subfigure}
    \hfill
    \begin{subfigure}[t]{0.45\textwidth}
        \centering
        \includegraphics[width=\linewidth]{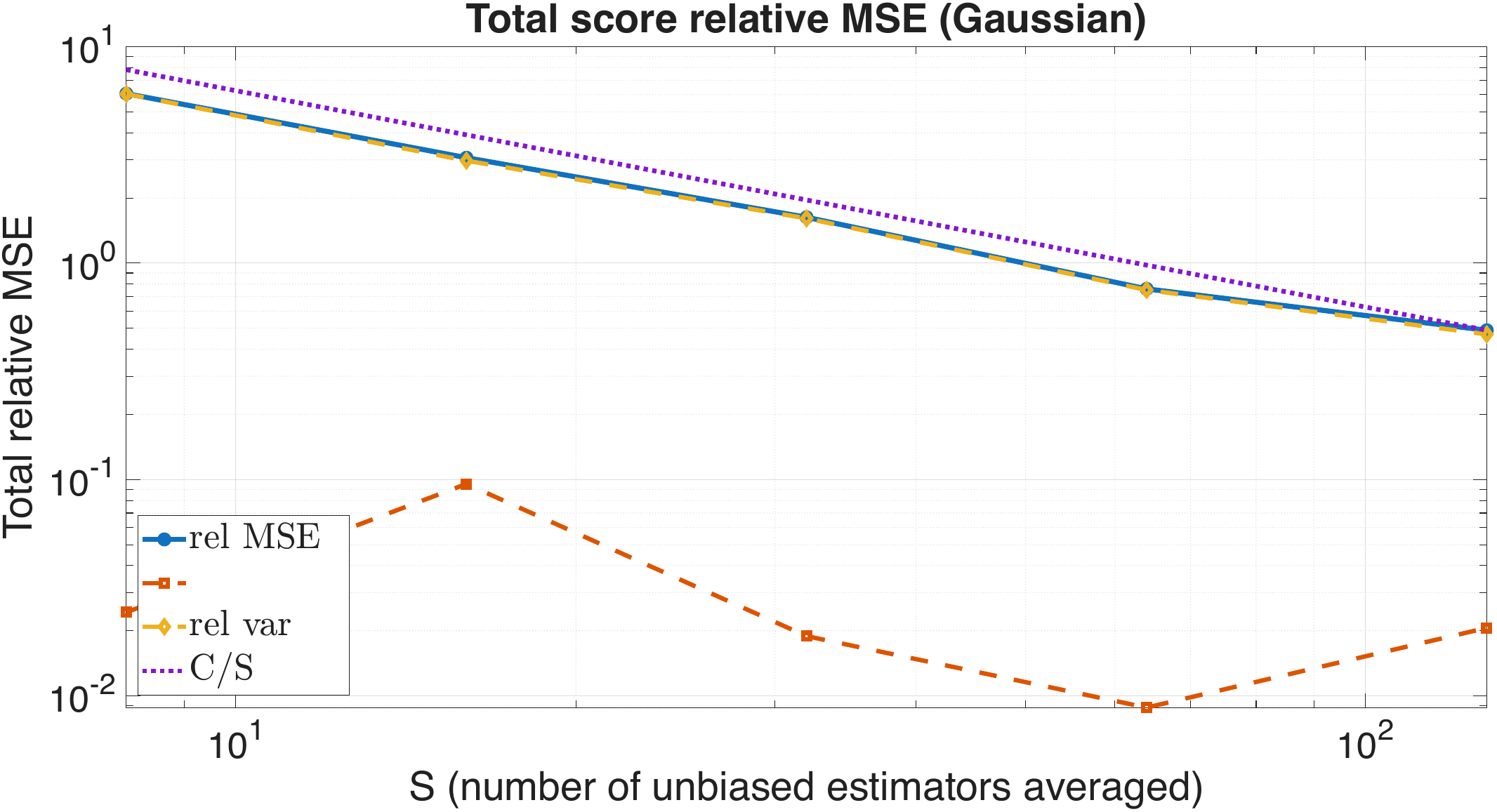}
        
        \label{fig:cn4}
    \end{subfigure}

    \caption{Estimation of the relative MSE, bias squared and variance of the estimated score function in terms of the average size $S$ for the Gaussian observations.}
    \label{fig:relMSE_gaus}
\end{figure}

\begin{figure}[htbp]
    \centering
    
    \begin{subfigure}[t]{0.45\textwidth}
        \centering
        \includegraphics[width=\linewidth]{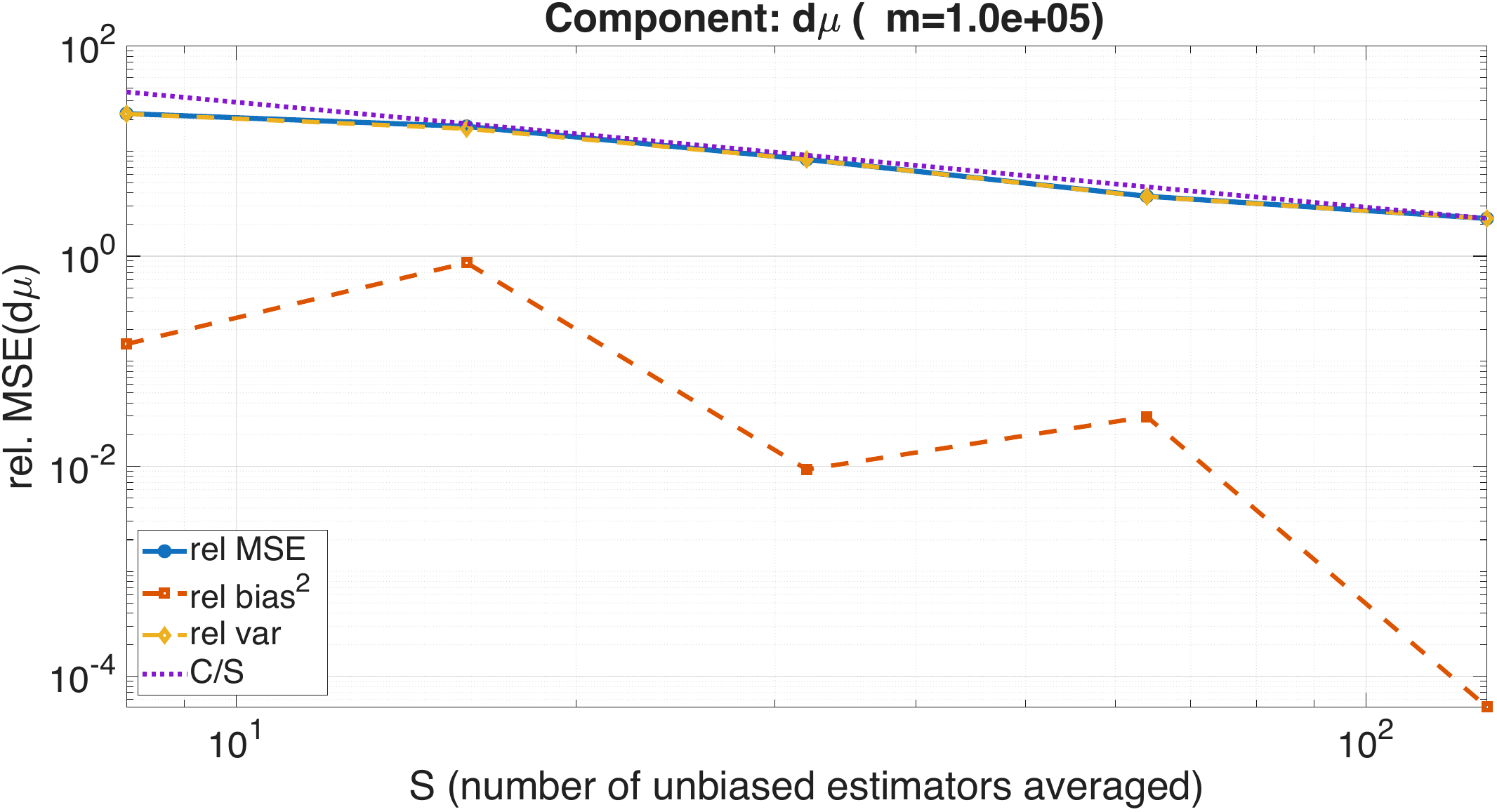}
        
        \label{fig:cn1}
    \end{subfigure}
    \hfill
    \begin{subfigure}[t]{0.45\textwidth}
        \centering
        \includegraphics[width=\linewidth]{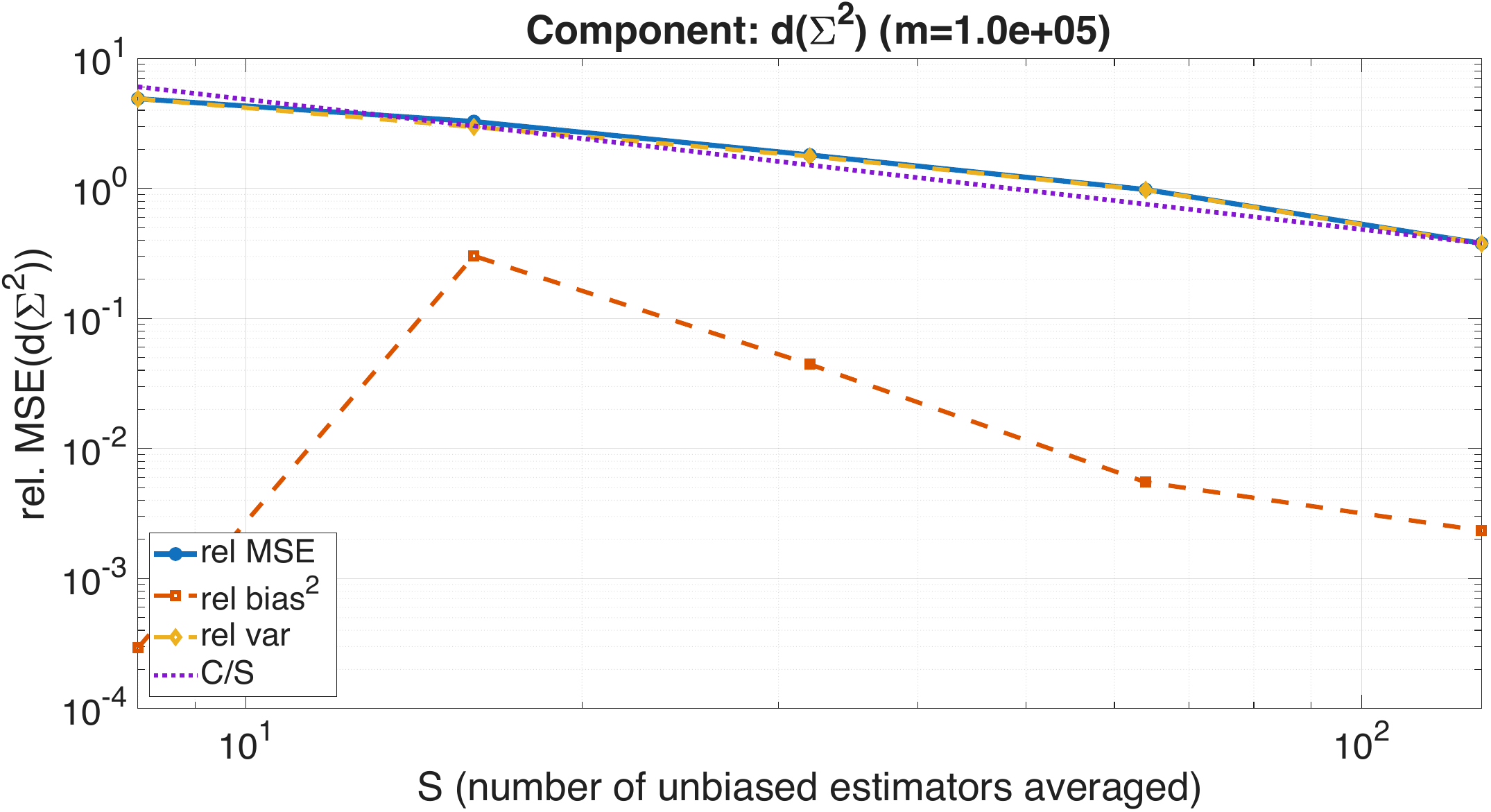}
        
        \label{fig:cn2}
    \end{subfigure}

    \vspace{0.5cm}

    \begin{subfigure}[t]{0.45\textwidth}
        \centering
        \includegraphics[width=\linewidth]{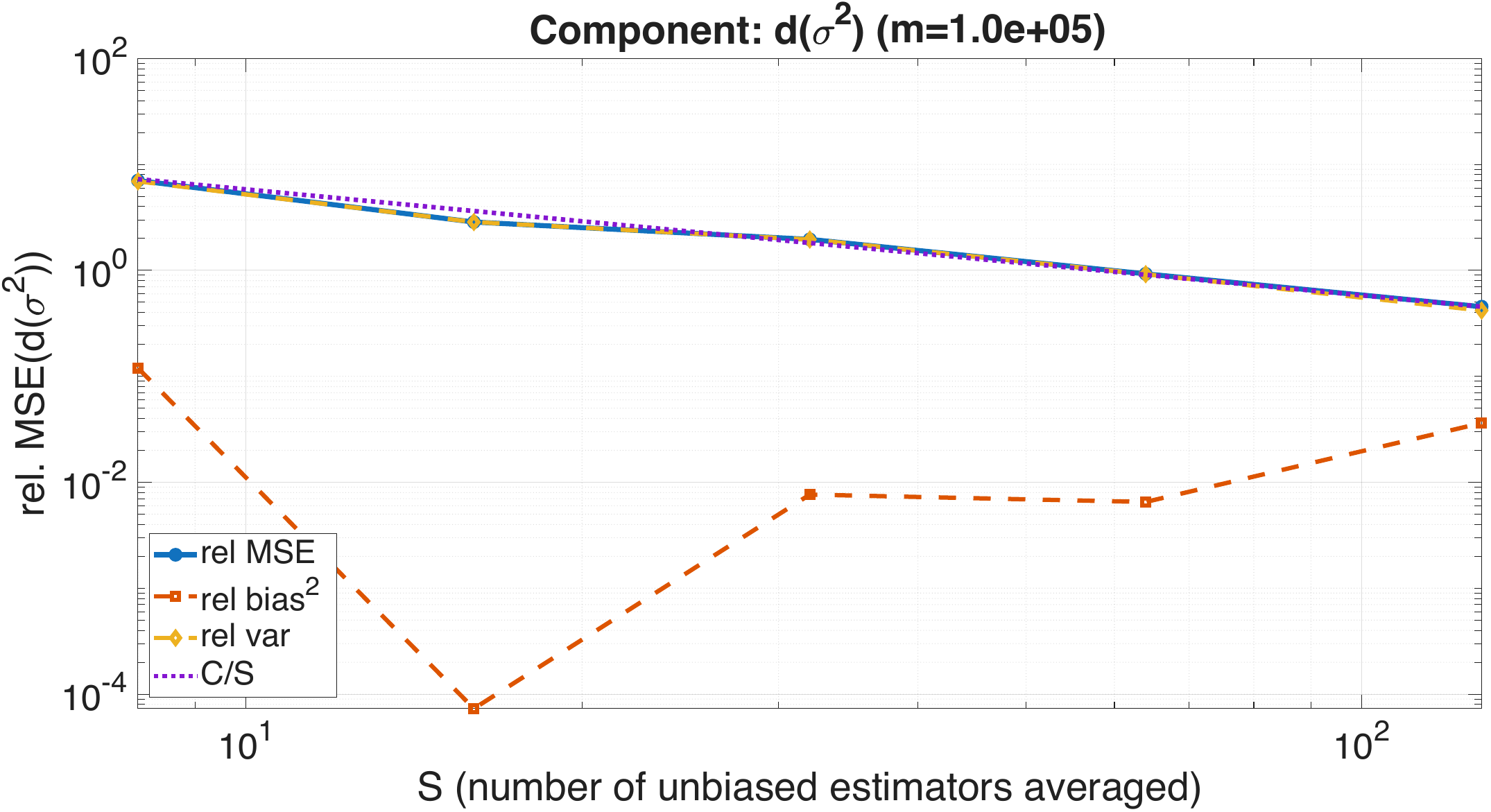}
        
        \label{fig:cn3}
    \end{subfigure}
    \hfill
    \begin{subfigure}[t]{0.45\textwidth}
        \centering
        \includegraphics[width=\linewidth]{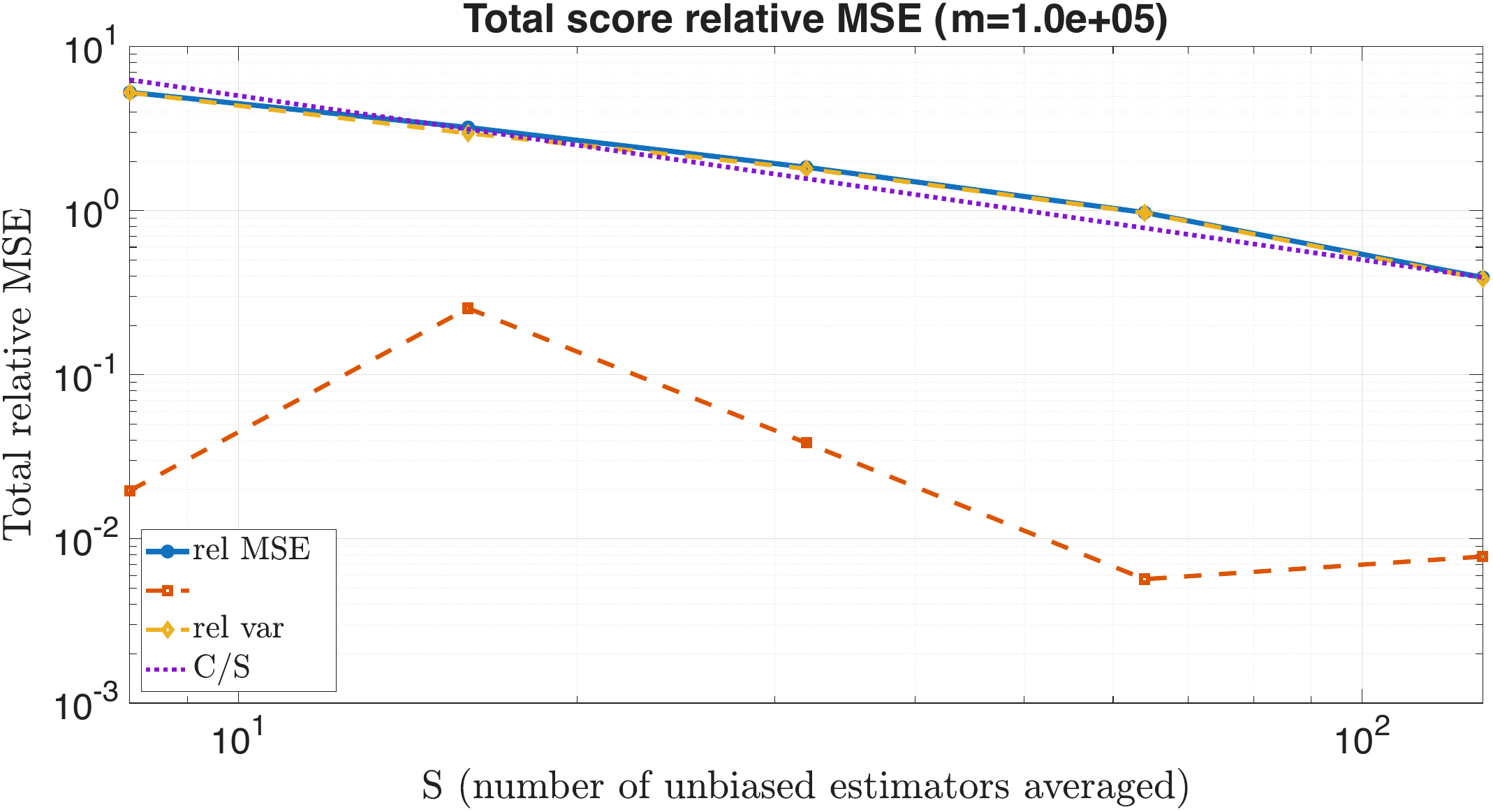}
        
        \label{fig:cn4}
    \end{subfigure}

    \caption{Estimation of the relative MSE, bias squared and variance of the estimated score function in terms of the average size $S$ for the Student-$t$ distributed observations with $m=10^5$.}
    \label{fig:relMSEtstud}
\end{figure}
Figure~\ref{fig:relMSE_mix} is analogous to Figure~\ref{fig:relMSEtstud} in that it corresponds to the case of Student-$t$ distributed observations, with the difference that here \(m = 5\). In this setting, the reference value of the score function is approximated by averaging a relatively large number of unbiased estimators. We observe that the same conclusions drawn from Figures~\ref{fig:relMSE_gaus} and~\ref{fig:relMSEtstud} continue to hold.
\begin{figure}[htbp]
    \centering
    
    \begin{subfigure}[t]{0.45\textwidth}
        \centering
        \includegraphics[width=\linewidth]{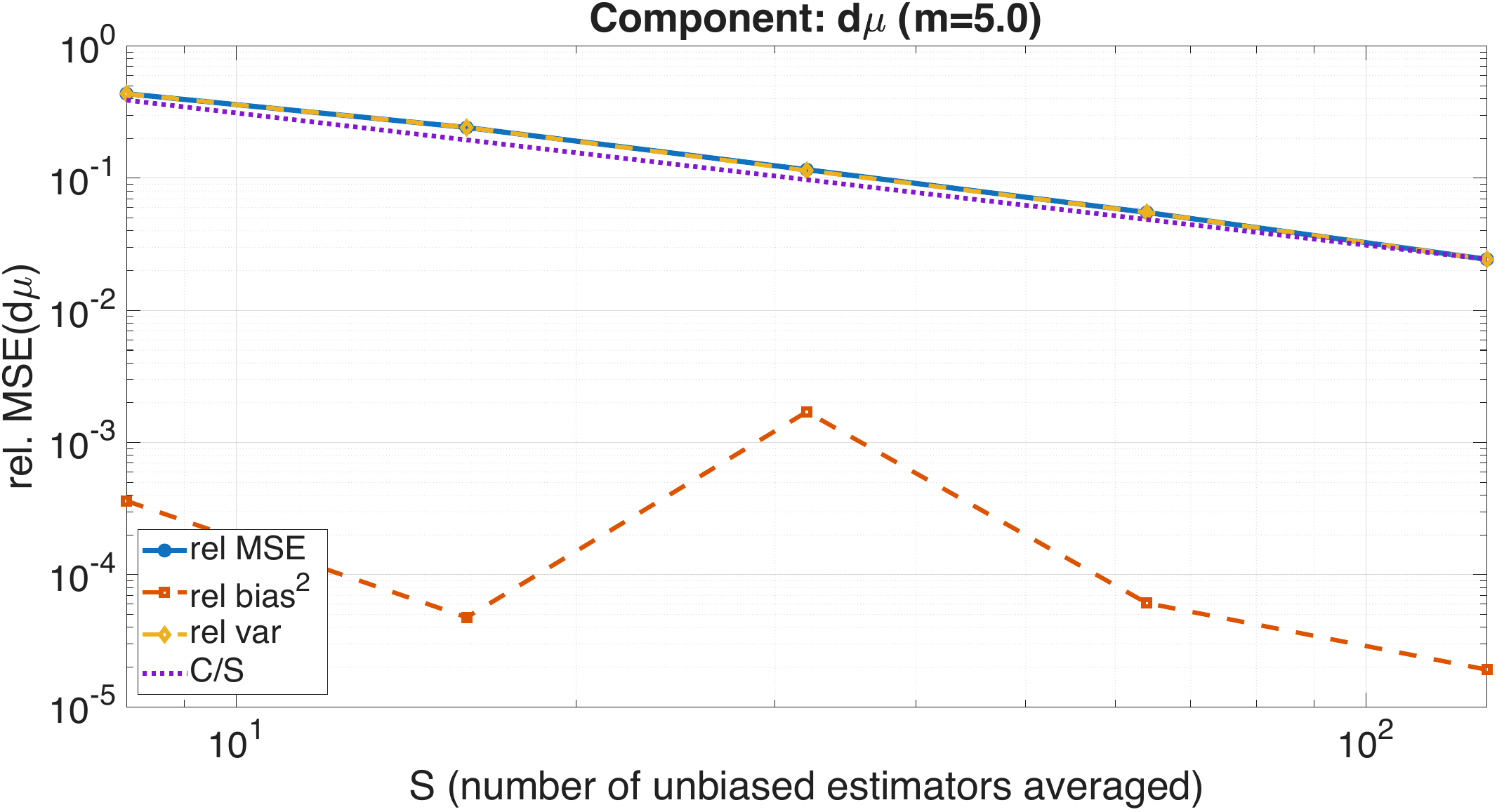}
        
        \label{fig:cnm1}
    \end{subfigure}
    \hfill
    \begin{subfigure}[t]{0.45\textwidth}
        \centering
        \includegraphics[width=\linewidth]{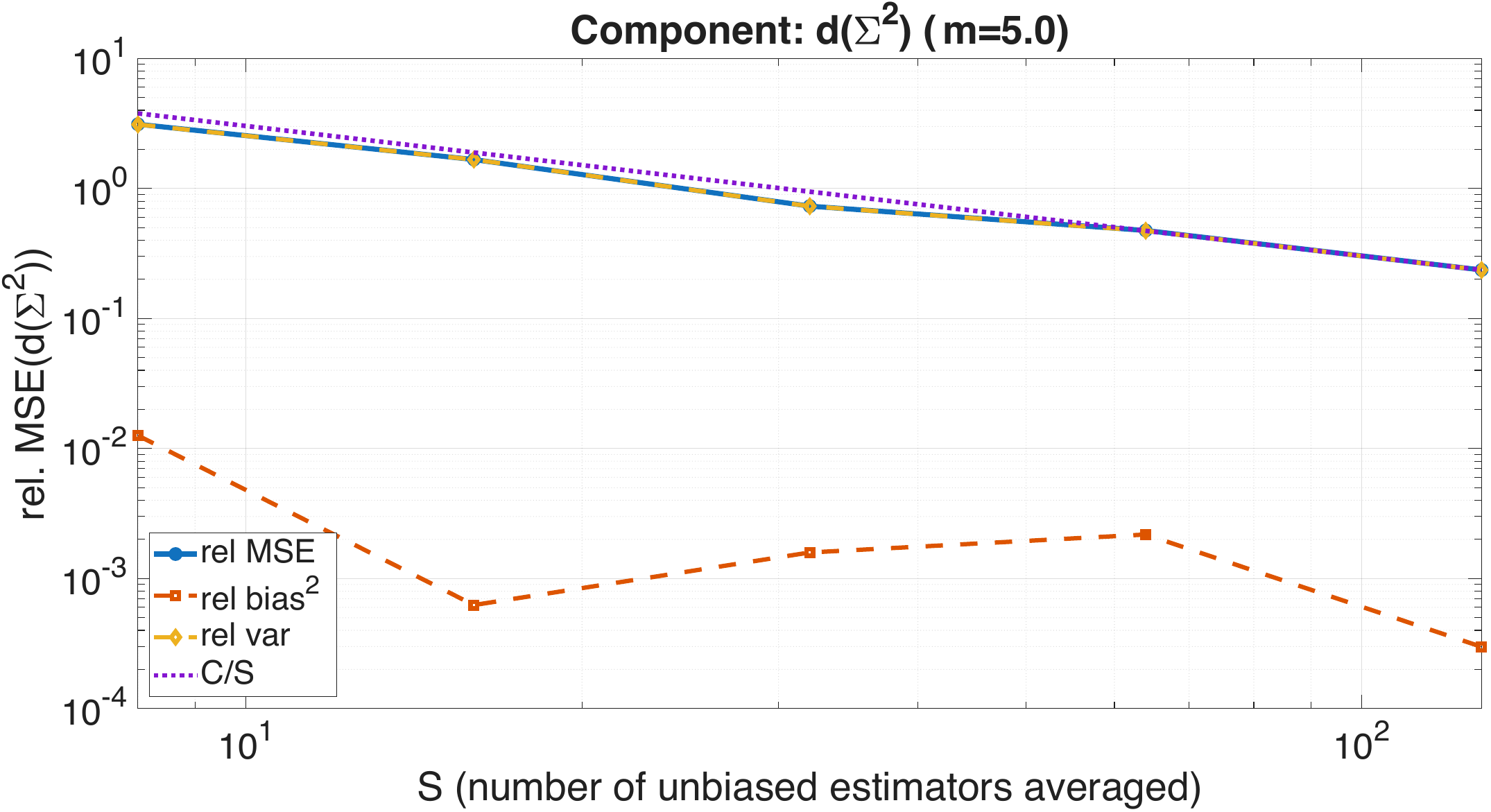}
        
        \label{fig:cnm2}
    \end{subfigure}

    \vspace{0.5cm}

    \begin{subfigure}[t]{0.45\textwidth}
        \centering
        \includegraphics[width=\linewidth]{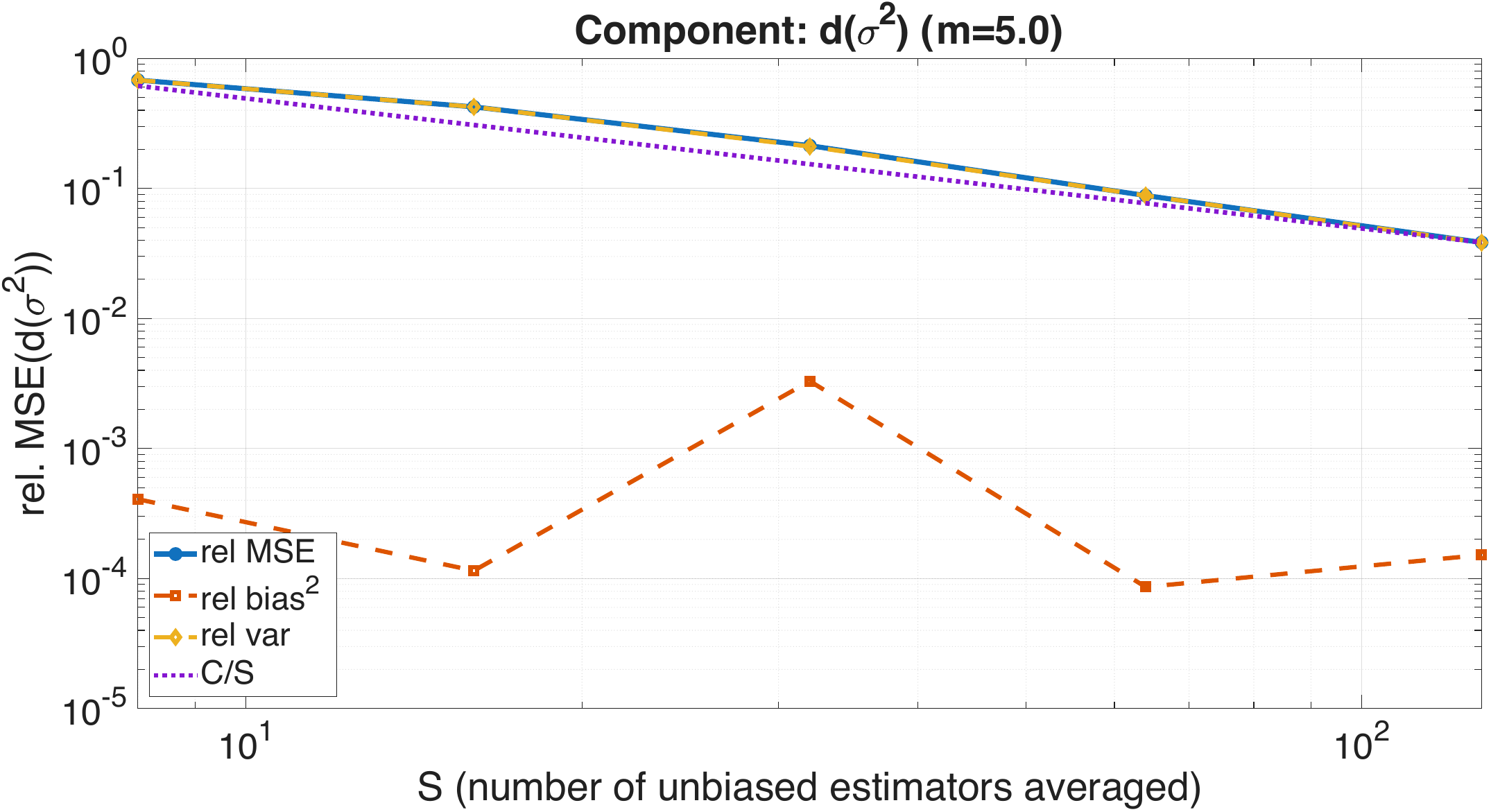}
        
        \label{fig:cnm3}
    \end{subfigure}
    \hfill
    \begin{subfigure}[t]{0.45\textwidth}
        \centering
        \includegraphics[width=\linewidth]{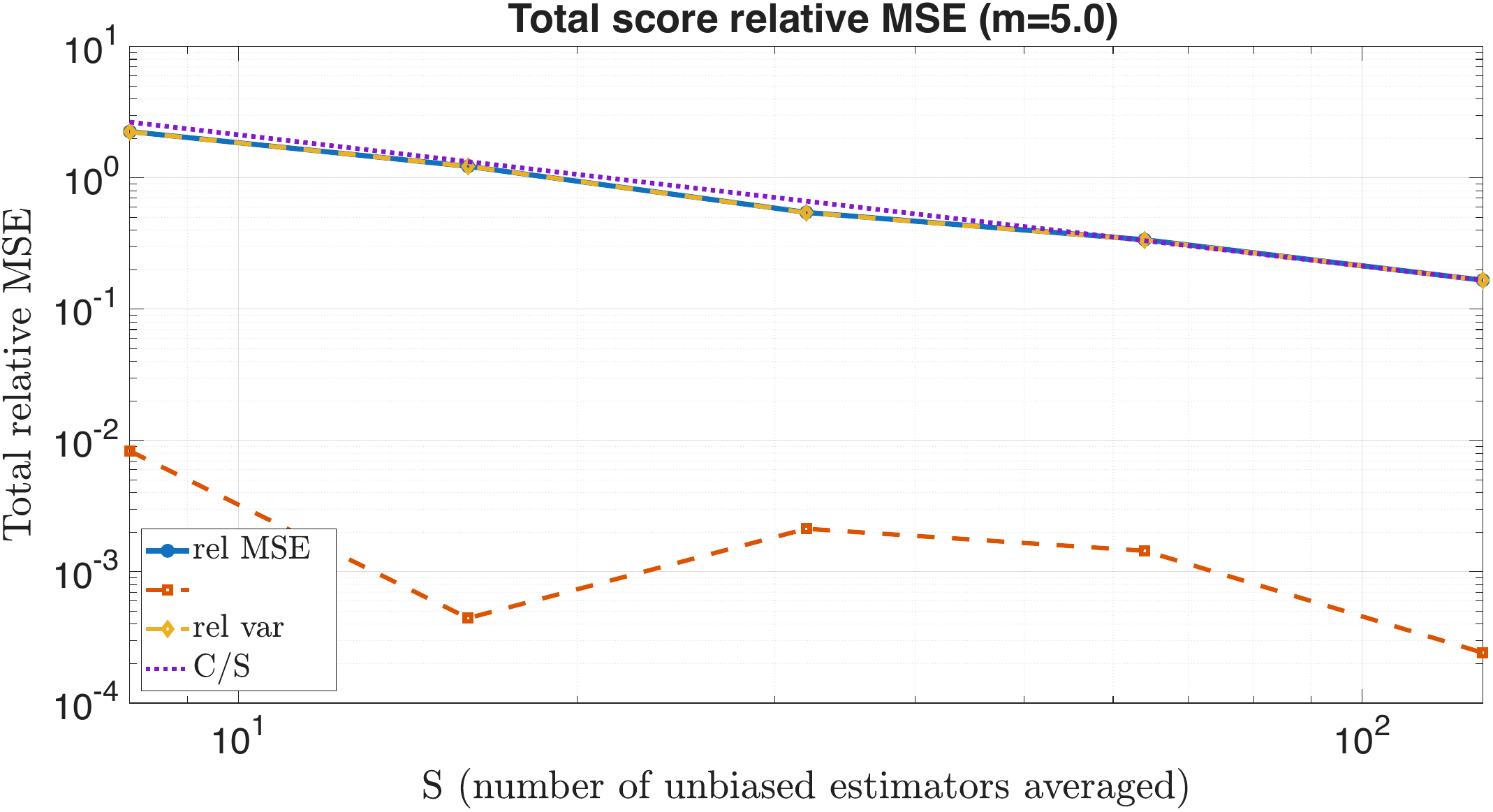}
        
        \label{fig:cnm4}
    \end{subfigure}

    \caption{Estimation of the relative MSE, bias squared and variance of the estimated score function in terms of the average size $S$ for the Student-$t$ distributed observations with $m=5$.}
    \label{fig:relMSE_mix}
\end{figure}

In the following, we consider the stochastic approximation method and display the numerical results in terms of the inferred static parameters. The data are simulated under the parameters of the model are $(\sigma,\mu,\Sigma) = (0.2, 0.7,0.25)$, $m=5$.  We use the particle Gibbs method with $N_{pf}=10$ and two iterations per update of the MSA method. $L_{\text{max}}=6$ with a number of $S=5$ unbiased estimators per step. 
The initial guesses are  $(\sigma_0,\mu_0,\Sigma_0)=(\sqrt{0.3},0.5,\sqrt{0.2})$. We use a logarithmic reparameterization $\tilde{\xi}=(\log({\sigma)},{\mu},\log({\Sigma}))$ of some of the variables of  $\xi$ in order to restrict the parametric space to positive standard deviations.   The step-size sequence is defined as
$\gamma_n = \gamma_0 (100 + n)^{-0.6}$ where $\gamma_0$ is 0.1 for the $\sigma$ update and 2 for the $(\mu,\Sigma)$ updates.

Figure \ref{fig:SA} shows how the MSE of the iterated parameters decreases in terms of the MSA iterations, demostrating the feasibility of the unbiased method. The MSE is estimated with a number \(C = 20\) independent stochastic approximation runs, additionally, the true solution is approximated using a biased MSA approximation with a significantly larger number of iterations $K\gg 10^3$ compared to the values displayed in Figure \ref{fig:SA}.

\begin{figure}[htbp]
    \centering
    \includegraphics[width=0.7\textwidth]{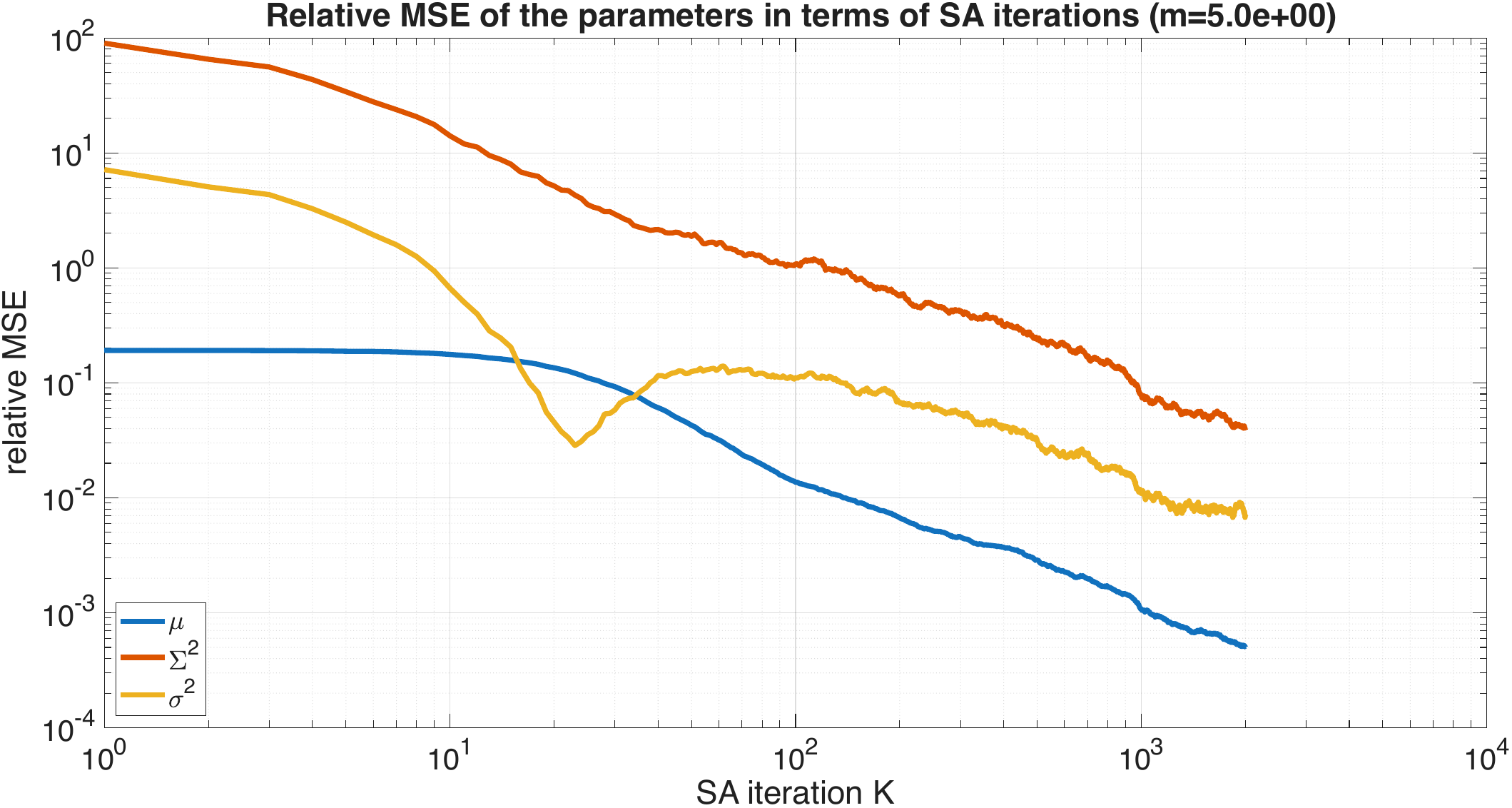}
    \caption{Estimated relative MSE of the component in terms of MSA iteration.}
    \label{fig:SA}
\end{figure}

\subsection{Portfolio selection}

\subsubsection{Framework}
\label{subsec:Framework}
We consider a model on $p$-dimensional stock returns observed from time $1$ to $\mathscr{T}\in\mathbb{N}$ associated with a latent process and static parameters. Denote the observed returns by $r_{1:\mathscr{T}}$, the latent process by $u_{1:\mathscr{T}}$, and the static parameters by $\theta$. We consider a full-factor Multivariate Stochastic Volatility (MSV) model:
\begin{equation}\label{eq:post_msv}
\pi(u_{1:\mathscr{T}},\theta \mid r_{1:\mathscr{T}}) \propto p(r_{1:\mathscr{T}} \mid u_{1:\mathscr{T}},\theta)\,p(u_{1:\mathscr{T}}\mid\theta)\,\overline{\pi}(\theta)
\end{equation}


where $p(r_{1:\mathscr{T}} \mid u_{1:\mathscr{T}},\theta)$ is the density of the returns given the latent process and the parameter,  $p(u_{1:\mathscr{T}}\mid\theta)$ the density of the latent process given the parameter and
$\overline{\pi}$ is the prior on $\theta$.
An example of such a model, which we use here, is given in \cite{mv_sv}.

Now, for $\mathscr{T}'$ \emph{future} log-returns, we would like to find an optimal portfolio. 
In the framework of Section \ref{sec:approach},  $Z=(\theta,U_{1:\mathscr{T}})$ and $X=
(R_{\mathscr{T}+1:\mathscr{T}+\mathscr{T}'},U_{\mathscr{T}+1:\mathscr{T}+\mathscr{T}'})$.
Let
$$
\omega(\xi) := \left(\frac{e^{\xi_1}}{\sum_{j=1}^p e^{\xi_j}},\dots,\frac{e^{\xi_p}}{\sum_{j=1}^p e^{\xi_j}}\right)^{\top}.
$$
We would like to find the portfolio $\omega(\xi)$, consisting of long-only positions, that maximizes
$$
\max_{\xi\in\mathbb{R}^p}\mathbb{E}[f(Z,\omega(\xi))]
$$
where the expectation is with respect to the posterior \eqref{eq:post_msv}, 
$$
f(z,\omega(\xi)) = \mathbb{E}[g(\theta,
U_{1:\mathscr{T}},R_{\mathscr{T}+1:\mathscr{T}+\mathscr{T}'},U_{\mathscr{T}+1:\mathscr{T}+\mathscr{T}'},\omega(\xi))\mid\theta,u_{1:\mathscr{T}}]
$$ 
$g$ is real-valued
and the conditional expectation is with respect to the predictive model; this is discussed below.
The function
$$
g(\theta,
U_{1:\mathscr{T}},R_{\mathscr{T}+1:\mathscr{T}+\mathscr{T}'},U_{\mathscr{T}+1:\mathscr{T}+\mathscr{T}'},\omega(\xi))
$$
$$
= \frac{1}{\mathscr{T}'}\sum_{t=\mathscr{T}+1}^{\mathscr{T}+\mathscr{T}'}
\left\{\omega(\xi)^{\top}R_t-\frac{\zeta}{2}\omega(\xi)^{\top}\left(R_{t}-\mathbb{E}[R_{t}\mid U_{1:\mathscr{T}},\theta]\right)\left(R_{t}-
\mathbb{E}[R_{t}\mid U_{1:\mathscr{T}},\theta]\right)^{\top}\omega(\xi)\right\}
$$
arises from portfolio wealth optimization under variance constraints using Lagrange multipliers, where $\zeta\in \mathbb{R}^+$ parameterizes the relation between volatility and returns of the solution.  Other approaches can be found, for example, in \cite{ps}. 
Note that $\mathbb{E}[R_{t}\mid U_{1:\mathscr{T}},\theta]$ is the expectation under the predictive model given
$Z$.  That is,  in the model to be considered we will have a structure that for any $n\in\mathbb{N}$
$$
p(r_{1:n},u_{1:n}|\theta) = \prod_{k=1}^n p(r_k|u_k,\theta) p(u_k|u_{k-1},\theta),
$$
so the conditional expectation of $r_t$ for any $t\geq\mathscr{T}+1$ is of the form
$$
\mathbb{E}[R_{t}\mid U_{1:\mathscr{T}},\theta] := \int r_t \left\{\prod_{k=\mathscr{T}+1}^t p(r_k|u_k,\theta) p(u_k|u_{k-1},\theta)
\right\}
du_{\mathscr{T}+1:t}dr_{\mathscr{T}+1:t}
$$
and $u_{\mathscr{T}},\theta$ are given.   In addition,  we have that 
\begin{align*}
\mathbb{E}[
g(Z,X,\omega(\xi))|z] =&
\int g(\theta,
u_{1:\mathscr{T}},r_{\mathscr{T}+1:\mathscr{T}+\mathscr{T}'},u_{\mathscr{T}+1:\mathscr{T}+\mathscr{T}'},\omega(\xi))
\\
&\left\{\prod_{k=\mathscr{T}+1}^{\mathscr{T}+\mathscr{T}'} p(r_k|u_k,\theta) p(u_k|u_{k-1},\theta)
\right\}\pi({u}_{1:\mathscr{T}},\mid r_{1:\mathscr{T}},{\theta})
du_{\mathscr{T}+1:\mathscr{T}+\mathscr{T}'}dr_{\mathscr{T}+1:\mathscr{T}+\mathscr{T}'}.
\end{align*}

We remark that the structure of the model is such that one can compute the inner-expectation, associated to the gradient,  unbiasedly using (standard) Monte Carlo methods and the outer expectation using MCMC as in Section \ref{sec:approach}.


\subsubsection{Model}

We present the full-factor MSV model \cite{mv_sv} in more detail.  We take for $t\in\mathbb{N}$
$$
\log(R_t+1)| B, V, F_t \stackrel{\text{ind}}{\sim} \mathcal{N}_p(BF_t,V)
$$
where 
$\mathcal{N}_{p}(\mu,\Sigma)$ is the $p-$dimensional Gaussian distribution of mean $\mu$ and covariance $\Sigma$, $B$ is $p\times K$ loading matrix that is unknown, $K\in\mathbb{N}$, $K\ll p$,  $F_t\in\mathbb{R}^K$ is a latent factor process and \( V \) is a stochastic diagonal matrix of variances.  Note that there are more latent processes
as we will now describe.  We will assume that for $t\in\mathbb{N}$
$$
F_t|\Sigma_t \stackrel{\text{ind}}{\sim} \mathcal{N}_{K}(0,\Sigma_t).
$$ 
The positive definite covariance matrix $\Sigma_t$ admits a spectral decomposition of the form \( \Sigma_t = P_t \Lambda_t P_t^\top \), where \( \Lambda_t \) is a diagonal eigenvalue matrix driven by a \(K\)-dimensional latent Gaussian process \(X_t\), with $\Lambda_{t}^{(i,i)}=\exp\left(X_{i,t}\right)$. Similarly, \( P_t \) is the eigenvector matrix, parameterized by a 
$K(K-1)/2-$dimensional Gaussian process \( \Psi_t \).

Let $\Psi_{ij,t}=\log\left(\pi/2-\Omega_{ij,t}\right)-\log\left(\pi/2+\Omega_{ij,t}\right)$ be the rotation parameter, for $1\leq i< j \leq p$, with $\Omega_{ij,t}\in (-\pi/2,\pi/2)$. The eigenvector matrix is characterized by the $K\times K$ rotation matrices $G_{ij,t}$ such that $P_t = \prod_{1 \leq i < j \leq p} G_{ij}(\Omega_{ij,t})$, where the rotation matrices are identity matrices with the following elements replaced:
\begin{align*}
&(G_{ij})^{(i,i)} = \cos(\Omega_{ij,t}), \quad (G_{ij})^{(j,j)} = \cos(\Omega_{ij,t}),\\    
&(G_{ij})^{(i,j)} = \sin(\Omega_{ij,t}), \quad (G_{ij})^{(j,i)} = -\sin(\Omega_{ij,t}).
\end{align*}
The latent processes $X_{1:\mathscr{T}+\mathscr{T}'}$ and $\Psi_{1:\mathscr{T}+\mathscr{T}'}$ follow independent Gaussian autoregressive dynamics
\begin{align}
\begin{split}
X_{i,t} &= x_{i,0} + \phi_i^x \big(X_{i,t-1} - x_{i,0}\big) + \sigma_i^x \eta_{i,t}, 
 \quad \quad \eta_{i,t} \sim \mathcal{N}(0,1), \quad \quad  i=1,\dots,K, \\
\Psi_{ij,t} &= \psi_{ij,0} + \phi_{ij}^\psi \big(\Psi_{ij,t-1} - \psi_{ij,0}\big) + \sigma_{ij}^\psi \xi_{ij,t}, \quad 
 \xi_{ij,t} \sim \mathcal{N}(0,1), \quad 1 \leq i < j \leq K.
\end{split}
\label{eq:gaussian_latent}
\end{align}
These dynamics depend on the parameters $ \theta_x = \{(\phi_i^x, x_{i,0}, \sigma_i^x)\}_{i=1}^{K} $ and  $ \theta_\psi = \{(\phi_{ij}^\psi, \psi_{ij,0}, \sigma_{ij}^\psi)\}_{1\leq i<j\leq K} $. We set $\theta=(B,V,\theta_x,\theta_\psi)$
and the entire latent process is $u_t=(F_t,X_t,\Psi_t)$. 
For further details on the full-factor MSV model, see \cite{mv_sv}.

%

The joint posterior of the latent process and parameters $(u_{1:\mathscr{T}},\theta)$ given the observed log-returns
$y_{1:\mathscr{T}}$ is
$$
\pi(u_{1:\mathscr{T}},\theta|y_{1:\mathscr{T}}) \propto \left\{\prod_{n=1}^{\mathscr{T}}p(y_n|B,V,F_n)
p(u_n|u_{n-1},\theta)\right\} \overline{\pi}(\theta)
$$ 
where the prior $\overline{\pi}(\theta)$ is as in \cite{mv_sv}. The MCMC algorithm used is that in \cite{mv_sv}.

\subsubsection{Numerical Results}
\label{subsubsec:num}

We now present numerical results for the proposed methodology applied to real financial data. We consider two equity markets: the STOXX Europe 600 and the Saudi Tadawul market. For each market, we construct a universe consisting of the most liquid stocks. We consider two experimental settings. In the first setting, we use a relatively small number of assets, \(p=20\), for both markets. This setting is used to illustrate the evolution of the objective function $F(\xi)$. In the second setting, we consider a larger-scale and more realistic portfolio allocation problem, using \(p=300\) assets for the STOXX Europe 600 and \(p=231\) assets for the Tadawul market. This second experiment evaluates the methodology in an online setting with sequential updates over a horizon of \(252\) trading days.

For the first setting, we consider a time horizon of length \(\mathscr{T}=494\), a forecasting horizon of \(\mathscr{T}'=5\), latent process dimension \(K=5\), and volatility-aversion parameter \(\zeta=20\). Further discussion regarding the choice of these parameters is provided in the context of the second experiment. The stochastic approximation algorithm is initialized using uniform portfolio weights, \(\omega_i=1/p\).

In Figure~\ref{fig:obj_f}, we display an approximation of the objective function ${F}(\xi)$, denoted by $\bar{F}(\xi)$, which is obtained by reusing the same Monte Carlo samples used to estimate the gradient at each stochastic approximation step. The figure shows the approximation of the objective function increasing throughout the stochastic approximation iterations before reaching a stationary regime. This behavior suggests that the variance of the gradient estimator remains sufficiently controlled to enable stable optimization in a moderately high-dimensional portfolio allocation problem.
\begin{figure}[htbp]
    \centering

    \begin{subfigure}[t]{0.48\textwidth}
        \centering
        \includegraphics[width=\linewidth]{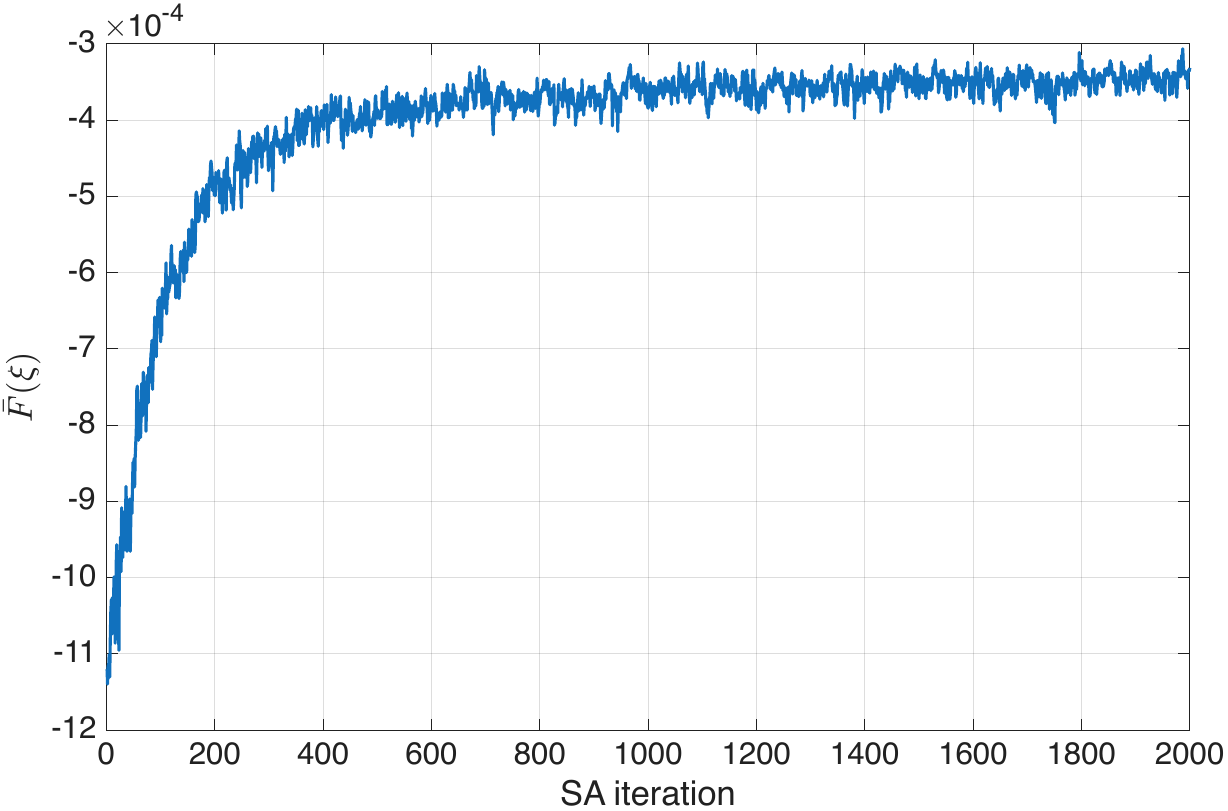}
        \caption{STOXX 600 Europe.}
        \label{fig:image1}
    \end{subfigure}
    \hfill
    \begin{subfigure}[t]{0.48\textwidth}
        \centering
        \includegraphics[width=\linewidth]{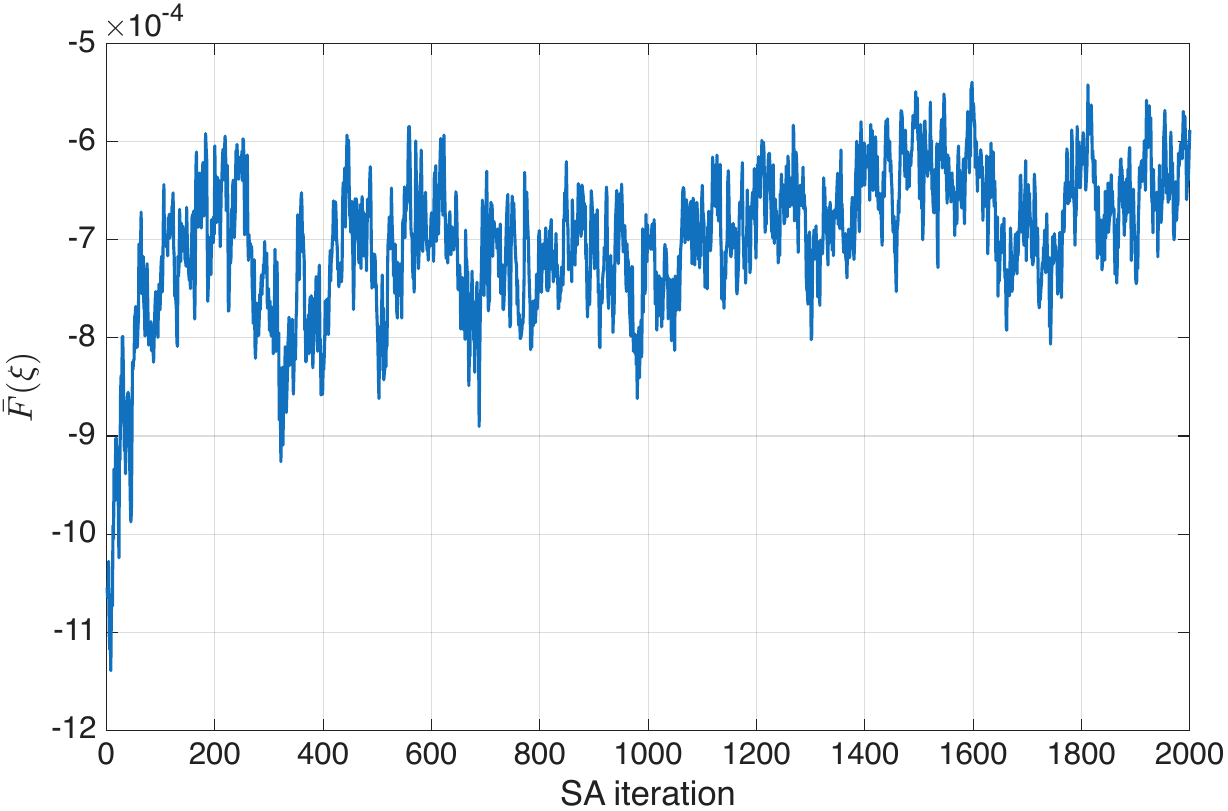}
        \caption{Tadawul.}
        \label{fig:image2}
    \end{subfigure}

    \caption{Approximation of the objective function $F(\xi)$ in terms of the number of iterations of the MSA method.}
    \label{fig:obj_f}
\end{figure}

For the second part, the algorithm is implemented in an online setting over an evaluation horizon of $252$ trading days, spanning the period from 2023-12-20 to 2024-12-13. Each dataset consists of approximately $730$ daily observations. At the initial rebalancing time, model parameters are estimated using a training window of $500$ observations. Subsequently, at each rebalancing date, the estimation window is expanded to include all available past data, and the model is updated accordingly. Portfolio rebalancing is performed every $\mathscr{T}'=1$ trading days. At each rebalancing date, portfolio weights are computed via a stochastic approximation scheme using the unbiased gradient estimator. The resulting weights are held fixed over the subsequent rebalancing interval.

The initial weights were chosen uniformly, \(\omega_i=1/p\). The latent dimension parameter was set to \(K=30\), as several backtesting experiments indicated that lower-dimensional specifications did not provide sufficient predictive power in terms of performance metrics such as the Sharpe ratio and final wealth. Similarly, we selected the volatility-aversion parameter \(\zeta=20\), since smaller values of \(\zeta\) tended to produce unstazzble portfolio allocations with excessive concentration on a small number of assets. In particular, the resulting portfolios often assigned disproportionately large weights to individual assets, making the strategy more sensitive to adverse market movements and reducing its ability to recover after losses.

The performance of the proposed unbiased MSV-MSA methodology is reported in Table~\ref{tab:performance_multi_market}. We evaluate performance using a range of standard financial metrics, including final wealth (FinalW), average percentage gain and loss, maximum drawdown (MDD), percentage of winning trades (WT), annualized return (Ann.R), annualized volatility (Ann.V), and the Sharpe ratio. As a benchmark, we consider the uniform ($1/p$) weight allocation strategy, which is widely used due to its simplicity and strong empirical performance.

\begin{table}[h!]
\centering
\small
\begin{tabular}{llcccccccc}
\toprule
Market & Strategy & FinalW & \% gain & \% loss & MDD & \% WT & Ann.R & Ann.V & Sharpe \\
\midrule

\multirow{2}{*}{\shortstack{STOXX \\Europe 600}} 
    & \shortstack{Unbiased\\ MSV-MSA} & 1.31 &    0.56 &   -0.49 &    6.62&    57.14 &    28.02 &    10.69&    2.62 \\
& Uniform     & 1.16&    0.48&  -0.49&  6.42 &    56.74 &    15.93&    9.82&    1.62\\

\midrule

\multirow{2}{*}{Tadawul} 
&\shortstack{Unbiased\\ MSV-MSA} &  1.39   &  0.54 &   -0.44 &   6.48   & 58.73   & 33.89 &   10.29 &   3.29 \\
& Uniform     &1.15  &  0.59 &   -0.63&    13.20&    56.34 &   15.48 &    13.26 &   1.16 \\

\bottomrule
\end{tabular}
\caption{Performance metrics for the unbiased MSV-MSA strategy and the uniform benchmark across the STOXX Europe 600 and Tadawul markets.}
\label{tab:performance_multi_market}
\end{table}

The results indicate consistent improvements of the unbiased MSV-MSA strategy relative to the uniform benchmark across both the STOXX Europe 600 and the Saudi Tadawul markets. In particular, the return-based metrics, including final wealth, annualized return, and Sharpe ratio, improve substantially under the proposed methodology.

Furthermore, these gains are not simply explained by increased portfolio volatility. For the STOXX Europe 600, the annualized volatility increases only moderately while the annualized return nearly doubles. In the Tadawul market, the proposed strategy simultaneously achieves higher returns, lower volatility, and significantly reduced maximum drawdown. Consequently, the Sharpe ratio improves considerably in both markets, increasing from \(1.62\) to \(2.62\) for the STOXX Europe 600 and from \(1.16\) to \(3.29\) for the Tadawul market.

Importantly, these results are obtained over an evaluation period exceeding one trading year rather than over short backtesting windows. This longer horizon helps mitigate the influence of transient market effects and provides stronger evidence of the robustness of the proposed methodology. Overall, the results demonstrate that unbiased stochastic approximation can be successfully integrated within high-dimensional MSV models to deliver meaningful improvements in risk-adjusted portfolio performance.

%
%

\subsubsection{Framework 2}

Consider a new ordering of the variables such that $\tilde{\theta}=(\theta_x,\theta_\psi)$, and $\chi=(B,V)$. The joint posterior-predictive distribution of the future returns, latent variables and parameters $({u}_{1:\mathscr{T}+\mathscr{T}'},{r}_{\mathscr{T}+1:\mathscr{T}+\mathscr{T}'},\chi,\tilde{\theta})$ given the observed log $r_{1:\mathscr{T}}$ is

\begin{align}
\begin{split}
\pi({u}_{1:\mathscr{T}+\mathscr{T}'},r_{\mathscr{T}+1:\mathscr{T}+\mathscr{T}'}, {\theta}|r_{1:\mathscr{T}})
& =\tilde{\pi}({u}_{1:\mathscr{T}+\mathscr{T}'},r_{\mathscr{T}+1:\mathscr{T}+\mathscr{T}'} ,\chi,\tilde{\theta}|r_{1:\mathscr{T}}) \\
&\propto \left\{\prod_{n=1}^{\mathscr{T}+\mathscr{T}'}p(r_n|{u}_n,\chi)p(u_n|u_{n-1},\tilde{\theta})\right\}p(\chi)\tilde{\pi}\left(\tilde{\theta}\right),
\label{eq:exp2}
\end{split}
\end{align}
where $p(u_{1}\mid u_0,\tilde{\theta})=p(u_{1}\mid \tilde{\theta})$ since $X_0$, $\Psi_0\in \tilde{\theta}$, and the prior distributions $p(\chi)\tilde{\pi}(\tilde{\theta})=\overline{\pi}(\theta)$ can be found in \cite{mv_sv}. Under this reparameterization the maximization problem is also reformulated as 
$$
\max_{\xi \in \mathbb{R}^p}\mathbb{E}\left[\tilde{f}(\widetilde{Z},\omega(\xi))\right]=\max_{\xi\in\mathbb{R}^p}\mathbb{E}[f(Z,\omega(\xi))]
$$

$$
\tilde{f}(\widetilde{z},\omega(\xi)) = \mathbb{E}[\tilde{g}(\tilde{\theta},
R_{\mathscr{T}+1:\mathscr{T}+\mathscr{T}'},U_{1:\mathscr{T}},\chi,\omega(\xi))\mid\tilde{\theta}]
$$ 
where the expectation $\mathbb{E}$ is w.r.t. the posterior \eqref{eq:exp2}. We choose $\tilde{Z}=\tilde{\theta}$ and $\tilde{X}=
(R_{\mathscr{T}+1:\mathscr{T}+\mathscr{T}'},{U}_{1:\mathscr{T}},\chi)$, these variables represent $Z$ and $X$ in section \ref{sec:approach}, we add the tilde symbol $\textasciitilde$ to differentiate them from variables established in subsection \ref{subsec:Framework}. The function $g$ is then
$$
\tilde{g}(\tilde{\theta},R_{\mathscr{T}+1:\mathscr{T}+\mathscr{T}'},{U}_{ 1:\mathscr{T}},\chi,\omega(\xi))
$$
$$
= \frac{1}{\mathscr{T}'}\sum_{t=\mathscr{T}+1}^{\mathscr{T}+\mathscr{T}'}
\left\{\omega(\xi)^{\top}R_t-\frac{\zeta}{2}\omega(\xi)^{\top}\left(R_{t}-\mathbb{E}\left[R_{t}\mid {U}_{1:\mathscr{T}},\chi,\tilde{\theta}\right]\right)\left(R_{t}-
\mathbb{E}\left[R_{t}\mid {U}_{1:\mathscr{T}},\chi,\tilde{\theta}\right]\right)^{\top}\omega(\xi)\right\}
$$
where $\mathbb{E}\left[R_{t}\mid {U}_{1:\mathscr{T}},\chi, \tilde{\theta}\right]$ is the predictive expectation conditioned on ${U}_{1:\mathscr{T}}$, $\chi$, and $\tilde{\theta}$ given by
$$
\mathbb{E}\left[R_{t}\mid {U}_{1:\mathscr{T}},\chi,\tilde{\theta}\right] := \int r_t \left\{\prod_{k=\mathscr{T}+1}^t p(r_k|{u}_k,\chi) p(u_k|u_{k-1},\tilde{\theta})
\right\}
du_{\mathscr{T}+1:t}dr_{\mathscr{T}+1:t}
$$
where $u_\mathscr{T}$, $\chi$ and $\tilde\theta$ are given. In addition,  we have  
\begin{align*}
&\mathbb{E}\left[
g(\tilde{Z},\tilde{X},\omega(\xi)|\tilde{z} \right] =
\int g(\tilde{\theta},r_{\mathscr{T}+1:\mathscr{T}+\mathscr{T}'},{u}_{1:\mathscr{T}},\chi,\omega(\xi))\\
&\left\{\prod_{k=\mathscr{T}+1}^{\mathscr{T}+\mathscr{T}'} p(r_k|{u}_k,\chi) p(u_k|u_{k-1},\tilde{\theta})
\right\}\tilde \pi({u}_{1:\mathscr{T}},\chi\mid r_{1:\mathscr{T}},\tilde{\theta})
du_{1:\mathscr{T}+\mathscr{T}'}dr_{\mathscr{T}+1:\mathscr{T}+\mathscr{T}'}d\chi.
\end{align*}

Using a Metropolis-within-Gibbs algorithm, we can obtain samples from a distribution $M_{\tilde{\theta}}(\tilde {X}^{n-1},\cdot)$  \cite{mv_sv} such that its invariant distribution is
\begin{align*}
\tilde{\pi}({u}_{1:\mathscr{T}},\chi|r_{1:\mathscr{T}},\tilde{\theta}) \propto \left\{\prod_{n=1}^{\mathscr{T}}p(r_n|{u}_n,\chi)p(u_n|u_{n-1},\tilde{\theta})\right\}p(\chi),
\end{align*}
similarly, we have access to samples from $K_{\tilde X}(\tilde{\theta}^{j-1},\cdot)$ with invariant distribution 
\begin{align*}
\left\{\prod_{n=1}^{\mathscr{T}}p(u_n|u_{n-1},\tilde{\theta})\right\}\tilde{\pi}(\tilde{\theta})
\end{align*}
or we can sample $\tilde{\theta}$ directly from the prior $\tilde{\pi}(\tilde{\theta})$. Thus, we can sample the inner expectation unbiasedly and subsequently estimate the outer expectation either sampling directly from $\tilde{\pi}\left(\tilde\theta\right)$ or in a MCMC fashion sampling from $K_{\tilde X}(\tilde{\theta}^{j-1},\cdot)$, as in section \ref{sec:approach}.

\subsection{Numerical results 2}

For Framework 2 we consider an almost identical setting to that of Subsubsection~\ref{subsubsec:num}. One difference is that in the second experiment we set the rebalancing period to \(\mathscr{T}'=5\) instead of \(\mathscr{T}'=1\). This choice reflects a trade-off between responsiveness and transaction costs: more frequent rebalancing allows the portfolio to adapt more rapidly to changes in the latent volatility structure, while less frequent rebalancing reduces trading costs. The selected frequency provides a practical compromise between these competing effects. Additionally, we use the truncated probability measure $ \widehat{\mathbb{P}}_{\mathtt{L}}(l)\propto (l+q)\,\log^{2}(l+q)\,2^{-l}$, with $l \in \{0,1,\dots,L_{\max}\}$, with $L_{\max}=6$ sufficiently large so that truncation effects are negligible in practice.

Figure~\ref{fig:obj_f_2} shows the evolution of the approximation of the objective function, while Table~\ref{tab:performance_multi_market_unbiased_2} summarizes the performance results for the second experiment. The conclusions are consistent with those obtained for Framework~1.
\begin{figure}[htbp]
    \centering

    \begin{subfigure}[t]{0.48\textwidth}
        \centering
        \includegraphics[width=\linewidth]{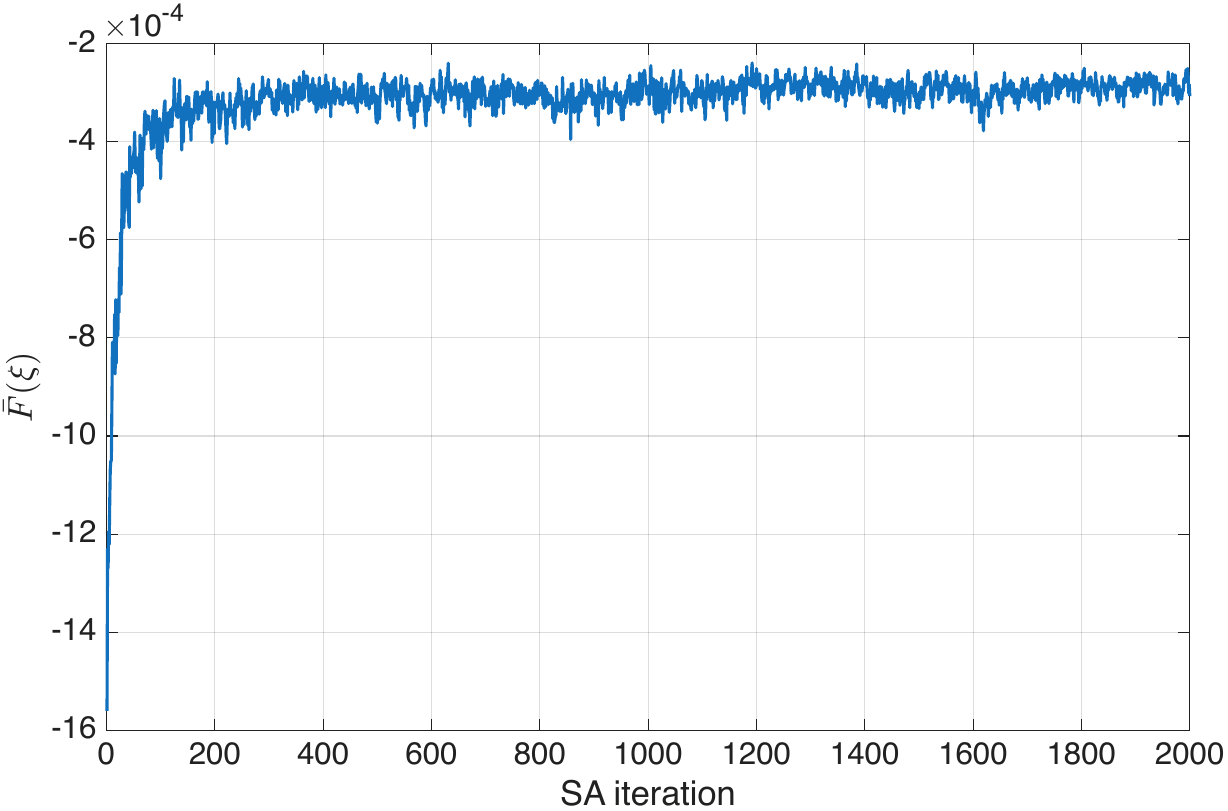}
        \caption{STOXX 600 Europe.}
        \label{fig:image1}
    \end{subfigure}
    \hfill
    \begin{subfigure}[t]{0.48\textwidth}
        \centering
        \includegraphics[width=\linewidth]{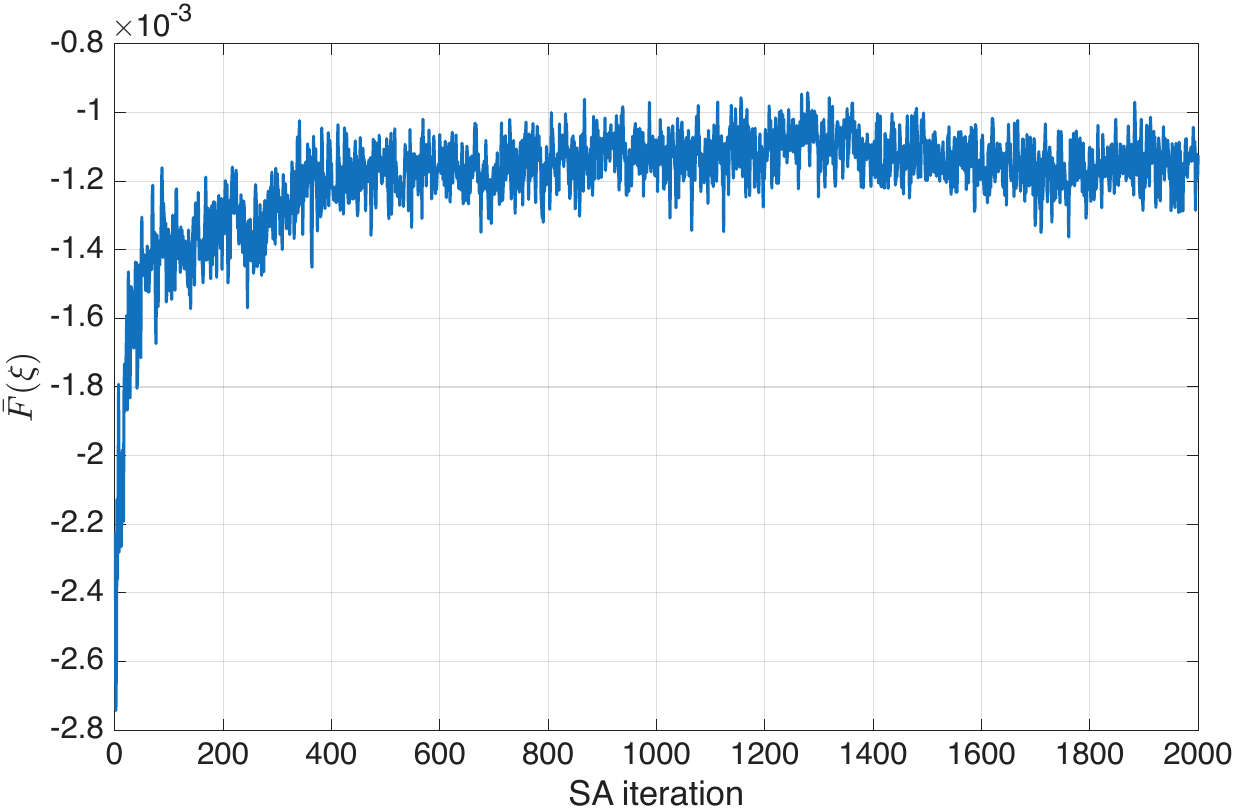}
        \caption{Tadawul.}
        \label{fig:image2}
    \end{subfigure}

    \caption{Approximation of the objective function $F(\xi)$ in terms of the number of iterations of the MSA method..}
    \label{fig:obj_f_2}
\end{figure}

\begin{table}[h!]
\centering
\small
\begin{tabular}{llcccccccc}
\toprule
Market & Strategy & FinalW & \% gain & \% loss & MDD & \% WT & Ann.R & Ann.V & Sharpe \\
\midrule

\multirow{2}{*}{\shortstack{STOXX \\Europe 600}} 
    & \shortstack{Unbiased\\ MSV-MSA} & 1.33 & 0.58 & -0.53 & 6.04 & 58.40  & 33.80 & 11.48 & 2.60 \\
& Uniform     & 1.17 & 0.49 & -0.50 & 6.43 & 56.80& 16.94 & 9.86 & 1.64 \\

\midrule

\multirow{2}{*}{Tadawul} 
&\shortstack{Unbiased\\ MSV-MSA} & 1.26&   0.60 & -0.65 &  11.17 &  59.60 &  25.81  &13.40&   1.78 \\
& Uniform     & 1.15 &  0.60 &  -0.63 &  13.21 &  56.00  &  15.48 &  13.31 &   1.15  \\

\bottomrule
\end{tabular}
\caption{Performance metrics for the unbiased MSV-MSA strategy under Framework 2.}
\label{tab:performance_multi_market_unbiased_2}
\end{table}

\appendix

\section{Proof of Proposition \ref{prop:1}}\label{app:prf_prop}

\begin{proof}
Throughout the proof $C$ is a generic constant and any important dependencies are mentioned if needed.
We begin by noting that by (A\ref{ass:1}) 1.~for any bounded and measurable $\varphi:\mathsf{X}\rightarrow\mathbb{R}$ we have for any $r\geq 1$ that
\begin{equation}\label{eq:prf1}
\mathbb{E}[|[\nu_z^{N_l-N_{l-1}}-\nu_z](\varphi)|^r]^{1/r} \leq \frac{C\|\varphi\|}{(N_l-N_{l-1}+1)^{\tfrac{1}{2}}}
\end{equation}
where $[\nu_z^{N_l-N_{l-1}}-\nu_z](\varphi)=\nu_z^{N_l-N_{l-1}}(\varphi)-\nu_z(\varphi)$,  
$\nu_z(\varphi)=\int_{\mathsf{X}}\varphi(x)\nu(x|z)dx$, $\|\varphi\|=\sup_{x\in\mathsf{X}}|\varphi(x)|$, $C$ is a constant  that does not depend on $l$ and $\mathbb{E}$ is the expectation w.r.t.~the law of the Markov chain in \textbf{MC}.
\eqref{eq:prf1} can be established using the Poisson equation for Markov chains,  see for example \cite{non_lin} and the references therein.  Then using the Minkowski inequality we have for any $l\in\mathbb{N}$
\begin{equation}\label{eq:prf2}
\mathbb{E}[|[\nu_z^{N_{1:l}}-\nu_z](\varphi)|^r]^{1/r} \leq C\|\varphi\|\sum_{p=1}^{l} \frac{(N_p-N_{p-1}+1)}{N_l(N_p-N_{p-1}+1)^{\tfrac{1}{2}}} \leq C\|\varphi\| 2^{-l/2}.
\end{equation}

To prove the result,  it is enough to show that for any $(i,j,s)\in\{1,\dots,k\}^3$ (see e.g.~\cite{ub_filt} and the references therein):
$$
\sum_{l=2}^{\infty} \frac{1}{\mathbb{P}_{\texttt{L}}(l)}
\mathbb{E}\Bigg[
\Big\{
\nu_z^{N_{1:l}}(\nabla g(z,\cdot,\xi)^{(i,j)})\nabla f\left(z,\nu_z^{N_{1:l}}(g(z,\cdot,\xi))\right)^{(s)} - 
$$
$$
\nu_z^{N_{1:l-1}}(\nabla g(z,\cdot,\xi))^{(i,j)}\nabla f\left(z,\nu_z^{N_{1:l-1}}(g(z,\cdot,\xi))\right)^{(s)}
\Big\}^2
\Bigg] <+\infty.
$$
We note that 
$$
\frac{1}{\mathbb{P}_{\texttt{L}}(1)}\mathbb{E}\left[
\left\{\nu_z^{N_1}(\nabla g(z,\cdot,\xi))^{(i,j)}\nabla f\left(z,\nu_z^{N_1}(g(z,\cdot,\xi)),z\right)^{(s)}\right\}^2\right]
$$
is trivially finite,  so this is why we do not consider it.  Now we have that
$$
\mathbb{E}\Bigg[
\Big\{
\nu_z^{N_{1:l}}(\nabla g(z,\cdot,\xi)^{(i,j)})\nabla f\left(z,\nu_z^{N_{1:l}}(g(z,\cdot,\xi))\right)^{(s)} - 
\nu_z^{N_{1:l-1}}(\nabla g(z,\cdot,\xi))^{(i,j)}\nabla f\left(z,\nu_z^{N_{1:l-1}}(g(z,\cdot,\xi))\right)^{(s)}
\Big\}^2
\Bigg]
$$
$$
\leq 2(T_1+T_2)
$$
where
\begin{align*}
T_1 & = \mathbb{E}\Bigg[
\Big\{
\nu_z^{N_{1:l}}(\nabla g(z,\cdot,\xi)^{(i,j)})\nabla f\left(z,\nu_z^{N_{1:l}}(g(z,\cdot,\xi))\right)^{(s)} - 
\nu_z(\nabla g(z,\cdot,\xi)^{(i,j)})\nabla f\left(z,\nu_z(g(z,\cdot,\xi))\right)^{(s)}\Big\}^2\Bigg] \\
T_2 & = \mathbb{E}\Bigg[
\Big\{
\nu_z^{N_{1:l-1}}(\nabla g(z,\cdot,\xi)^{(i,j)})\nabla f\left(z,\nu_z^{N_{1:l-1}}(g(z,\cdot,\xi))\right)^{(s)} - 
\nu_z(\nabla g(z,\cdot,\xi)^{(i,j)})\nabla f\left(z,\nu_z(g(z,\cdot,\xi))\right)^{(s)}\Big\}^2\Bigg].
\end{align*}
We note that bounding $T_1$ and $T_2$,  which is what we will do,  is essentially the same,  so we only give the proof for $T_1$.  We have that
$$
T_1 \leq 2(T_3+T_4)
$$
where 
\begin{align*}
T_3 & = \mathbb{E}\Bigg[
\Big\{
\{\nu_z^{N_{1:l}}(\nabla g(z,\cdot,\xi)^{(i,j)})-\nu_z(\nabla g(z,\cdot,\xi)^{(i,j)})\}\nabla f\left(z,\nu_z^{N_{1:l}}(g(z,\cdot,\xi))\right)^{(s)}\Big\}^2\Bigg] \\
T_4 & =  
\mathbb{E}\Bigg[
\Big\{
\nu_z(\nabla g(z,\cdot,\xi)^{(i,j)})
\{\nabla f\left(z,\nu_z^{N_{1:l}}(g(z,\cdot,\xi))\right)^{(s)}-
\nabla f\left(z,\nu_z(g(z,\cdot,\xi))\right)^{(s)}\}\Big\}^2\Bigg].
\end{align*}
For $T_3$ one can use (A\ref{ass:1}) 2.~along with \eqref{eq:prf2} to deduce that
$$
T_3 \leq C 2^{-l}.
$$
For $T_3$ one can use (A\ref{ass:1}) 2. ~and (A\ref{ass:1}) 2-3 ~along with the $C_2-$inequality $k$ times and \eqref{eq:prf2} to deduce that
$$
T_4 \leq C 2^{-l}.
$$
From here the proof is easily completed and this concludes the argument.
\end{proof}

\end{document}